\documentclass{jnsao}
\usepackage[utf8]{inputenc}
\usepackage[finnish,english]{babel}
\usepackage[svgnames]{xcolor}
\usepackage{tikz, pgfplots}
\usepgfplotslibrary{colorbrewer}
\usepackage{enumitem}
\usepackage[nameinlink,capitalise]{cleveref}
\usepackage{nicematrix}
\usepackage{booktabs}

\definecolor{hrefcolor}{rgb}{0.0,0.4,0.7}
\definecolor{citecolor}{rgb}{0.0,0.35,0.2}
\definecolor{structure}{rgb}{0.09,0.09,0.44}
\hypersetup{
    colorlinks=true,
    linkcolor=citecolor,
    citecolor=citecolor,
    filecolor=hrefcolor,
    urlcolor=hrefcolor,
    pdfencoding=auto,
    hypertexnames=false,
}

\newtheorem{assumption}[definition]{Assumption}
\counterwithin{assumption}{section} %

\crefname{assumption}{Assumption}{Assumptions}

\manuscriptcopyright{}
\manuscriptlicense{}

\usepackage{float}
\usepackage{algpseudocode}

\floatstyle{ruled}
\floatname{algorithm}{Algorithm}
\newfloat{algorithm}{tbp!}{loa}
\numberwithin{algorithm}{section}

\def\tripleVert{\|\kern-1pt|}
\def\onorm#1{\tripleVert #1 \tripleVert}
\def\doublelangle{\langle\kern-2pt\langle}
\def\doublerangle{\rangle\kern-2pt\rangle}
\def\oiprod#1#2{\doublelangle #1, #2 \doublerangle}

\tikzset{notestyleraw/.append style={align=justify}}

\title{Linearly convergent bilevel optimization with single-step inner methods}
\shorttitle{Bilevel optimization with single-step inner methods}

\author{%
    Ensio Suonperä\thanks{Department of Mathematics and Statistics, University of Helsinki, Finland. \mbox{\email{ensio.suonpera@helsinki.fi}}}
    \and
    Tuomo Valkonen\thanks{ModeMat, Escuela Politécnica Nacional, Quito, Ecuador \emph{and} Department of Mathematics and Statistics, University of Helsinki, Finland. \email{tuomo.valkonen@iki.fi}, \orcid{0000-0001-6683-3572}}
}

\acknowledgements{%
    This research has been supported by the Academy of Finland grants 314701, 320022, and 345486.
}

\date{2022-05-10 (revised 2023-05-29)}

\newcommand{\term}{\emph}

\newcommand{\field}[1]{\mathbb{#1}}
\newcommand{\N}{\mathbb{N}}

\newcommand{\R}{\field{R}}
\newcommand{\extR}{\overline \R}

\newcommand{\norm}[1]{\|#1\|}

\newcommand{\abs}[1]{|#1|}

\newcommand{\inv}[1]{#1^{-1}}
\newcommand{\grad}{\nabla}
\newcommand{\freevar}{\,\boldsymbol\cdot\,}
\newcommand{\Union}\bigcup
\newcommand{\Isect}\bigcap
\newcommand{\union}\cup
\newcommand{\isect}\cap
\newcommand{\bigunion}\bigcup
\newcommand{\bigisect}\bigcap

\newcommand{\defeq}{:=}

\newcommand{\downto}{\searrow}

\newcommand{\subdiff}{\partial}

\DeclareMathOperator*{\argmin}{arg\,min}

\DeclareMathOperator{\bd}{bd}

\DeclareMathOperator{\Dom}{dom}

\DeclareMathOperator{\diag}{diag}

\DeclareMathOperator{\prox}{prox}
\DeclareMathOperator{\proj}{proj}

\newcommand{\iprod}[2]{\langle #1,#2\rangle}

\def \weaktostarSym{\setbox0=\hbox{$\rightharpoonup$}\rlap{\hbox 
        to\wd0{\hss\raise1ex\hbox{$\scriptscriptstyle{*\,}$}\hss}}\box0}
    
\def\linear{\mathbb{L}}

\def\extR{\overline \R}

\let\phi=\varphi
\let\epsilon=\varepsilon

\DeclareMathOperator{\Id}{Id}

\def\nexxt#1{#1^{k+1}}
\def\this#1{#1^k}
\let\opt\widehat

\def\Vfun{\mathscr{V}}
\def\AlphaSpace{\mathscr{A}}

\begin{document}

\maketitle

\begin{abstract}
    We propose a new approach to solving bilevel optimization problems, intermediate between solving full-system optimality conditions with a Newton-type approach, and treating the inner problem as an implicit function.
    The overall idea is to solve the full-system optimality conditions, but to precondition them to alternate between taking steps of simple conventional methods for the inner problem, the adjoint equation, and the outer problem.
    While the inner objective has to be smooth, the outer objective may be nonsmooth subject to a prox-contractivity condition.
    We prove \emph{linear convergence} of the approach for combinations of gradient descent and forward-backward splitting with exact and inexact solution of the adjoint equation.
    We demonstrate good performance on learning the regularization parameter for anisotropic total variation image denoising, and the convolution kernel for image deconvolution. 
\end{abstract}

\section{Introduction}

Two general approaches are typical for the solution of the bilevel optimization problem
\begin{equation}
    \label{eq:bilevel-problem}
    \min_{\alpha\in\AlphaSpace}~J(S_u(\alpha)) + R(\alpha)
    \quad\text{with}\quad
    S_u(\alpha) \in \argmin_{u\in U} F(u; \alpha)
\end{equation}
in Hilbert spaces $\AlphaSpace$ and $U$.
The first, familiar from the treatment of general mathematical programming with equilibrium constraints (MPECs), is to write out the
Karush--Kuhn--Tucker conditions for the whole problem in a suitable form, and to apply a Newton-type method or other nonlinear equation solver to them \cite{bard1982explicit,allende2012solving,fliege2021gauss,jiang2013application,reyes2021bilevel}.

The second approach, common in the application of \eqref{eq:bilevel-problem} to inverse problems and imaging \cite{falk1995bilevel,delosreyes2014learning,kunisch2012bilevel,holler2018bilevel,hintermueller2017optimal,sherry2020learning,calatroni2014dynamic}, treats the solution mapping $S_u$ as an implicit function. Thus it is necessary to (i) on each outer iteration $k$ solve the inner problem $\min_u F(u; \alpha^k)$ near-exactly using an optimization method of choice; (ii) solve an adjoint equation to calculate the gradient of the solution mapping; and (iii) use another optimization method of choice on the outer problem $\min_\alpha J(S_u(\alpha))$ with the knowledge of $S_u(\alpha^k)$ and $\grad S_u(\alpha^k)$.
The inner problem is therefore generally assumed to have a unique solution, and the solution map to be differentiable. An algorithm for nonsmooth inner problems has been developed in \cite{liu2020generic}, while \cite{reyes2021optimality} rely on proving directional Bouligand differentiability for otherwise nonsmooth problems.

The challenge of the first “whole-problem” approach is to scale it to large problems, typically involving the inversion of large matrices.
The difficulty with the second “implicit function” approach is that the inner problem needs to be solved several times, which can be expensive. Solving the adjoint equation also requires matrix inversion.
The variant in \cite{ehrhadt2021inexact} avoids this through derivative-free methods for the outer problem. It also solves the inner problem to a low but controlled accuracy.

In this paper, by preconditioning the implicit-form first-order optimality conditions, we develop an intermediate approach more efficient than the aforementioned, as we demonstrate in the numerical experiments of \cref{sec:numerical}. It can be summarized as (i) take only \emph{one} step of an optimization method on the inner problem, (ii) perform a \emph{cheap} operation to advance towards the solution of the adjoint equation, and, finally, (iii) using this approximate information, take one step of an optimization method for the outer problem. Repeat.

The preconditioning, which we introduce in detail in \cref{sec:methods},
is based on insight from the derivation of the primal-dual proximal splitting of \cite{chambolle2010first} as a preconditioned proximal point method \cite{he2012convergence,tuomov-proxtest,clason2020introduction}.
We write the optimality conditions for \eqref{eq:bilevel-problem} as the inclusion $0 \in H(x)$ for a set-valued $H$, where $x=(u,p,\alpha)$ for an adjoint variable $p$.
The basic proximal point method then iteratively solves $x^{k+1}$ from
\[
    0 \in H(x^{k+1}) + (x^{k+1}-x^k).
\]
This can be as expensive as solving the original optimality condition. The idea then is to introduce a preconditioning operator $M$ that \emph{decouples} the components of $x$---in our case $u$, $p$ and $\alpha$---such that each component can be solved in succession from
\[
    0 \in H(x^{k+1}) + M(x^{k+1}-x^k).
\]
Gradient steps can be handled through nonlinear preconditioning \cite{tuomov-proxtest,clason2020introduction}, as we will see in \cref{sec:methods} when we develop the approach in detail along with two more specific algorithms, the FEFB (Forward--Exact--Forward-Backward) and the FIFB (Forward--Inexact--Forward-Backward). In \cref{sec:convergence} we prove their local linear convergence under a second-order growth condition on the composed objective $J \circ S_u$, and other more technical conditions. The proof is based on the “testing” approach developed in \cite{tuomov-proxtest} and also employed extensively in \cite{clason2020introduction,tuomov-firstorder}.
Finally, we evaluate the numerical performance of the proposed schemes on imaging applications in \cref{sec:numerical}, specifically the learning of a regularization parameter for total variation denoising, and the convolution kernel for deblurring. Since the purpose of these experiments is a simple performance comparison between different algorithms, instead of real applications, we only use a single training sample of various dimensions, as explained in \cref{sec:numerical}.

Intermediate approaches, some reminiscent of ours, have recently also been developed in the machine learning community.
Our approach, however, allows a non-smooth function $R$ in the outer problem \cref{eq:bilevel-problem}. Moreover, to our knowledge, our work is the first to show linear convergence for a fully “single-loop” algorithm.
To be more precise, the STABLE \cite{chen2021single}, TTSA \cite{hong2020two}, FLSA \cite{li2022fully}, MRBO, VRBO \cite{yang2021provably}, and SABA \cite{dagreou2022framework} are “single-loop” algorithms such as ours, taking only a single step towards the solution of the inner problem on each outer iteration.
The STABLE requires solving the adjoint equation exactly, as does our first approach, the FEFB.
The others use a Neumann series approximation for the adjoint equation.
Our second approach, the FIFB, takes a simple step reminiscent of gradient descent for the adjoint equation.
The TSSA and STABLE obtain sublinear convergence of the outer iterates $\{\this\alpha\}_{k \in \N}$ assuming strong convexity (second-order growth) of both the inner and outer objective.
For the SABA similar linear convergence is claimed with the outer strong convexity replaced by a Polyak-{\L}ojasiewicz inequality.
Without either of those assumptions, the theoretical results on the aforementioned methods from the literature are much weaker, and generally only show various forms of “stall” of the iterates at a sublinear rate, or the ergodic convergence of the gradient $\grad_{\alpha}[J\circ S_u](\alpha^k)$ of the composed objective to zero. Such modes of convergence say very little about the convergence of function values to optimum or the iterates to a solution.

In the context of not fully single-loop algorithms, the AID, ITD \cite{ji2021bilevel}, AccBio \cite{ji2022lower}, and ABA \cite{ghadimi2018approximation} take a fixed (small) number of inner iterations for each outer iteration.
The AID and ITD only sublinear convergence of the composed gradient is claimed. For the ABA and AccBio linear convergence of outer function values is claimed under strong convexity of both the inner and outer objectives.

\subsection*{Fundamentals and applications}

Fundamentals of MPECs and bilevel optimization are treated in the books \cite{luo1996mathematical,aussel2017generalized,dempe2006foundations,dempe2015bilevel}.
An extensive literature review up to 2018 can be found in \cite{dempe2018bilevel}, and recent developments in \cite{dempe20202bilevel}. Optimality conditions for bilevel problems, both necessary and sufficient, are developed in, e.g., \cite{ye1995optimality,zemkoho2016solving,dempe2012bilevel,mehlitz2021sufficient,bai2022directional}. A more limited type of “bilevel” problems only constrains $\alpha$ to lie in the set of minimisers of another problem. Algorithms for such problems are treated in \cite{sabach2017firstorder,shehu2019inertial}.

Bilevel optimization has been used for \term{learning} regularization parameters and forward operators for inverse imaging problems. With total variation regularization in the inner problem, the parameter learning problem in its most basic form reads \cite{delosreyes2014learning}
\[
    \min_{\alpha}~\frac{1}{2}\norm{S_u(\alpha)-b}^2 + R(\alpha)
    \quad\text{with}\quad
    S_u(\alpha) = \argmin_{u\in U} \frac{1}{2}\norm{A_\alpha u-z}^2 + \alpha_1\norm{\grad u}_1.
\]
This problem finds the best possible $\alpha$ for reconstructing the “ground truth” image $b$ from the measurement data $z$, which may be noisy and possibly transformed and only partially known through the \term{forward operator} $A_\alpha$, mapping images to measurements.
To generalize to multiple images, the outer problem would sum over them and corresponding inner problems \cite{calatroni2014dynamic}.
Multi-parameter regularization is discussed in \cite{tuomov-tgvlearn}, and natural conditions for $\alpha>0$ in \cite{tuomov-interior}.\footnote{An error in \cite[Lemma 10]{tuomov-interior} requires some conditions therein to be taken “in the limit” as $t \downto 0$.}
In other works, the forward operator $A_\alpha$ is learned for blind image deblurring \cite{hintermuller2015bilevel} or undersampling in magnetic resonance imaging \cite{sherry2020learning}.
In \cite{kunisch2012bilevel} regularization kernels are learned, while  \cite{reyes2021optimality,chambolle2021learning} study the learning of optimal discretisation schemes.
To circumvent the non-differentiability of $S_u$, \cite{ochs2016techniques} replace the inner problem with a fixed number of iterations of an algorithm.
Their approach has connections to the learning of deep neural networks.

Bilevel problems can also be seen as leader--follower or \term{Stackelberg games}: the outer problem or agent leads by choosing $\alpha$, and the inner agent reacts with the best possible $u$ for that $\alpha$.
Multiple-agent \term{Nash equilibria} may also be modeled as bilevel problems.
Both types of games can be applied to financial markets and resource use planning; we refer to the the aforementioned books \cite{luo1996mathematical,aussel2017generalized,dempe2006foundations,dempe2015bilevel} for specific examples.

\subsection*{Notation and basic concepts}
\label{par:notation}

We write $\linear(X; Y)$ for the space of bounded linear operators between the normed spaces $X$ and $Y$ and $\Id$ for the identity operator.
Generally $X$ will be Hilbert, so we can identify it with the dual $X^*$.

For $G \in C^1(X)$, we write $G'(x) \in X^*$ for the Fr\'{e}chet derivative at $x$, and $\grad G(x) \in X$ for its Riesz presentation, i.e., the gradient.
For $E \in C^1(X; Y)$, since $E'(x) \in \linear(X; Y)$, we use the Hilbert adjoint to define $\grad E(x) \defeq G'(x)^* \in \linear(Y; X)$.
Then the Hessian $\grad^2 G(x) \defeq \grad[\grad G](x) \in \linear(X; X)$.
When necessary we indicate the differentiation variable with a subscript, e.g., $\grad_u F(u, \alpha)$.
For convex $R: X \to \extR$, we write $\Dom R$ for the effective domain and $\subdiff R(x)$ for the subdifferential at $x$.
With slight abuse of notation, we identify $\subdiff R(x)$ with the set of Riesz presentations of its elements.
We define the proximal operator as $\prox_R(x) \defeq \argmin_z \frac{1}{2}\norm{z-x}^2 + R(z)=\inv{(\Id+\subdiff R)}(x)$.

We write $\iprod{x}{y}$ for an inner product, and $B(x,r)$ for a closed ball in a relevant norm $\norm{\freevar}$.
For self-adjoint positive semi-definite $M\in \linear(X; X)$ we write $\norm{x}_{M} \defeq \sqrt{\iprod{x}{x}_{M}} \defeq \sqrt{\iprod{Mx}{x}}.$
Pythagoras' or \emph{three-point identity} then states
\begin{equation}\label{3-point-identity}
    \iprod{x-y}{x-z}_{M} = \frac{1}{2}\norm{x-y}^2_M - \frac{1}{2}\norm{y-z}^2_M + \frac{1}{2}\norm{x-z}^2_M \qquad \text{for all } x,y,z\in X.
\end{equation}
We extensively use Young's inequality
\[
    \iprod{x}{y} \leq \frac{a}{2}\norm{x}^2 + \frac{1}{2a}\norm{y}^2 \qquad \text{for all } x,y\in X,\, a > 0.
\]
We sometimes apply operations on $x \in X$ to all elements of a set $A \subset X$, writing $\iprod{x+A}{z} \defeq \{\iprod{x+a}{z}
\mid a \in A \}$, and for $B \subset \R$, writing $B \ge c$ if $b \ge c$ for all $b \in B.$

\section{Proposed methods}
\label{sec:methods}

We now present our proposed methods for \eqref{eq:bilevel-problem}.
They are based on taking a single gradient descent step for the inner problem, and using forward-backward splitting for the outer problem.
The two methods differ on how an “adjoint equation” is handled.
We present the algorithms and assumptions required to prove their convergence in \cref{subsec:Algorithm:FEFB,subsec:Algorithm:FIFB} after deriving optimality conditions and the adjoint equation in \cref{subsec:optimality cond}.
We prove convergence in \cref{sec:convergence}.

\subsection{Optimality conditions}
\label{subsec:optimality cond}

Suppose $u\mapsto F(u;\alpha)\in C^2(U)$ is proper, coercive, and weakly lower semicontinuous for each \term{outer variable} $\alpha \in \Dom R \subset \AlphaSpace$.
Then the direct method of the calculus of variations guarantees the inner problem $\min_u F(u; \alpha)$ to have a solution.
If, further, $u\mapsto F(u;\alpha)$ is strictly convex, the solution is unique so that the solution mapping $S_u$ from \eqref{eq:bilevel-problem} is uniquely determined.

Suppose further that $F, \grad F$ and $S_u$ are Fr\'{e}chet differentiable.
Writing $T(\alpha) \defeq (S_u(\alpha), \alpha)$, Fermat's principle and $S_u(\tilde\alpha) \in \argmin_u F(u; \tilde\alpha)$ then show that
\begin{equation}\label{eq:inner-S_u-oc}
    [\grad_{u} F\circ T] (\alpha)= \grad_{u} F(S_u(\alpha); \alpha)  =0
\end{equation}
for $\alpha$ near $\tilde\alpha$.
Therefore, the chain rule for Fr\'{e}chet differentiable functions yields
\[
    0=\grad_{\alpha}[\grad_{u} F \circ T](\alpha) = \grad_{\alpha}S_u(\alpha) \grad_{u}^2 F(T(\alpha)) + \grad_{\alpha u} F(T(\alpha)).
\]
That is, $p=\grad_{\alpha}S_u(\alpha)$ solves for $u=S_u(\alpha)$ the \term{adjoint equation}
\begin{equation}
    \label{eq:p-row-oc}
    0=p \grad_{u}^2 F(u, \alpha) + \grad_{\alpha u} F(u, \alpha).
\end{equation}
We introduce the corresponding solution mapping for the \term{adjoint variable} $p$,
\begin{equation}
    \label{def:S_p}
    S_p(u,\alpha) := - \grad_{\alpha u} F(u; \alpha) \left (\grad_u^2 F(u; \alpha)\right )^{-1}.
\end{equation}
We will later make assumptions that ensure that $S_p$ is well-defined.

Since $S_u: \AlphaSpace \to U$, the Fréchet derivative $S_u'(\alpha) \in \linear(\AlphaSpace; U)$ and the Hilbert adjoint $\grad_\alpha S_u(\alpha) \in \linear(U; \AlphaSpace)$ for all $\alpha$. Consequently $p \in \linear(U; \AlphaSpace)$, but we will need $p$ to lie in  an inner product space.
Assuming $\AlphaSpace$ to be a \emph{separable} Hilbert space, we introduce such structure
\begin{subequations}
\label{eq:p-space}
\begin{equation}
    P=(\linear(U; \AlphaSpace), \oiprod{\freevar}{\freevar})
\end{equation}
by using a countable orthonormal basis $\{\phi_i\}_{i\in I}$ of $\AlphaSpace$ to define the inner product
\begin{equation}
    \label{def:p-inner_prod}
	 \oiprod{p_1}{p_2}
     \defeq
     \sum_{i\in I} \iprod{p_1^* \phi_i}{p_2^* \phi_i}
     =
     \sum_{i\in I} \iprod{ \phi_i}{p_1 p_2^* \phi_i}.
     \quad (p_1, p_2 \in \linear(U; \AlphaSpace)).
\end{equation}
\end{subequations}
We briefly study this inner product and the induced norm $\onorm{\freevar}$ in \cref{sec:separable}.

By the sum rule for Clarke subdifferentials (denoted $\subdiff_C$) and their compatibility with convex subdifferentials and Fréchet differentiable functions \cite{clarke1990optimization}, we obtain
\[
    \subdiff_C (J \circ S_u+R)(\opt{\alpha})
    =
    \grad_{\alpha}(J \circ S_u)(\opt{\alpha}) +  \partial R(\opt{\alpha})
    =
    \grad_{\alpha}S_u(\opt{\alpha})\grad_{u}J(S_u(\opt{\alpha})) +  \partial  R(\opt{\alpha}).
\]
The Fermat principle for Clarke subdifferentials then furnishes the necessary optimality condition
\begin{equation}
    \label{eq:outer-oc}
    0 \in \grad_{\alpha}(J \circ S_u)(\opt{\alpha}) +  \partial R(\opt{\alpha})= \grad_{\alpha}S_u(\opt{\alpha})\grad_{u}J(S_u(\opt{\alpha})) +  \partial  R(\opt{\alpha}).
\end{equation}

We combine \cref{eq:inner-S_u-oc,eq:p-row-oc,eq:outer-oc} as the inclusion 
\begin{gather}
    \label{eq:main_optimality_condition}
    0 \in H(\opt{u}, \opt{p}, \opt{\alpha})
\shortintertext{with}
    \label{eq:H-def}
    H(u,p,\alpha) := \begin{pmatrix}
        \grad_{u} F(u; \alpha)  \\
        p \grad_{u}^2 F(u; \alpha) + \grad_{\alpha u} F(u; \alpha)  \\
        p\grad_{u}J(u) +  \partial R(\alpha)
    \end{pmatrix}
    \quad\text{for all}\quad
    u\in U, p\in P, \alpha\in\AlphaSpace.
\end{gather}
This is the optimality condition that our proposed methods, presented in \cref{subsec:Algorithm:FEFB,subsec:Algorithm:FIFB}, attempt to satisfy.
We generally abbreviate
\[
    x=(u,p,\alpha),
    \quad
    \opt x=(\opt u,\opt p,\opt\alpha),
    \quad
    \text{etc.}
\]

\subsection{Algorithm: forward--exact--forward-backward}
\label{subsec:Algorithm:FEFB}

Our first strategy for solving \cref{eq:main_optimality_condition} takes just a \emph{single} gradient descent step for the inner problem, solves the adjoint equation exactly, and then takes a forward-backward step for the outer problem.
We call this \cref{alg:FEFB} the FEFB (forward--exact--forward-backward).

\begin{algorithm}
    \caption{Forward--exact--forward-backward (FEFB) method }
    \label{alg:FEFB}
    \begin{algorithmic}[1]
        \Require
        Functions $R$, $J$, $F$ as in \cref{ass:FEFB:main}.
        Step length parameters $\tau,\sigma>0$.
        \State Pick an initial iterate $(u^{0}, \alpha^{0}) \in U \times \AlphaSpace.$
        \For{$k \in \N$}
        \State $u^{k+1} \defeq u^{k} - \tau \grad_{u} F(u^{k}; \alpha^{k})$
        \Comment{inner gradient step}
        \State $p^{k+1} \defeq - \grad_{\alpha u} F(u^{k+1}; \alpha^{k})\left (\grad_{u}^2 F(u^{k+1}; \alpha^{k})\right )^{-1}$
        \Comment{adjoint solution}
        \State $\alpha^{k+1} \defeq \prox_{ \sigma R}\left (\alpha^{k} - \sigma ( p^{k+1}\grad_{u}J(u^{k+1}))\right )$
        \Comment{outer forward-backward step}
        \EndFor
    \end{algorithmic}
\end{algorithm}

Using $H$ defined in \eqref{eq:H-def}, \cref{alg:FEFB} can be written implicitly as solving
\begin{equation}
    \label{eq:FEFB:implicit}
    0 \in H_{k+1}(x^{k+1}) + M(x^{k+1} - x^k)
\end{equation}
for $x^{k+1} = (u^{k+1}, p^{k+1}, \alpha^{k+1})$, where, with $x=(u, p, \alpha)$,
\begin{subequations}
\begin{equation}
    \label{eq:H_k+1-def}
    H_{k+1}(x)
        \defeq
        H(x)
        + \begin{pmatrix}
            \grad_u F(u^k; \alpha^k) - \grad_u F(u; \alpha)  \\
            p \grad_{u}^2 F(u; \alpha^k) + \grad_{\alpha u}F(u; \alpha^k) - \left (p \grad_{u}^2 F(u; \alpha) + \grad_{\alpha u}F(u; \alpha)\right ) \\
            0
        \end{pmatrix},
\end{equation}
and the \term{preconditioning operator} $M\in \mathbb{L}(U\times P \times \AlphaSpace; U\times P \times \AlphaSpace)$ is
\begin{equation}
    \label{eq:mkplus1}
    M \defeq
    \diag(
        \tau^{-1}\Id, 0, \sigma^{-1}\Id
    ).
\end{equation}
\end{subequations}
The “nonlinear preconditioning” applied to $H$ to construct $H_{k+1}$ shifts iterate indices such that a forward step is performed instead of a proximal step; compare \cite{tuomov-proxtest,clason2020introduction}.

We next state essential structural, initialisations, and step length assumptions.
We start with a contractivity condition needed for the proximal step with respect to $R$.

\begin{assumption}
    \label{ass:contractivity}
    Let $R: \AlphaSpace \to \extR$ be convex, proper, and lower semicontinuous.
    We say that $R$ is \term{locally prox-$\sigma$-contractive at $\opt\alpha \in \AlphaSpace$ for $q \in \AlphaSpace$} (within $A \subset \Dom R$) if there exist $ C_R > 0$ and a neighborhood $A \subset \Dom R$ of $\opt\alpha$ such that, for all $\alpha \in A$,
    \[
        \norm{D_{\sigma R}(\alpha)-D_{\sigma R}(\opt\alpha)} \le \sigma C_R \norm{\alpha-\opt\alpha}
        \quad\text{for}\quad
        D_{\sigma R}(\alpha) \defeq \prox_{\sigma R}(\alpha - \sigma q)-\alpha.
    \]
\end{assumption}

We verify \cref{ass:contractivity} for some common cases in \cref{sec:prox-contractivity}.
When applying the assumption to to $\opt\alpha$ satisfying \eqref{eq:main_optimality_condition}, we will take $q = -\opt p\grad_u J(\opt u) \in \subdiff R(\opt\alpha)$.
Then $D_{\sigma R}(\opt\alpha)=0$ by standard properties of proximal mappings.
The results for nonsmooth functions in \cref{sec:prox-contractivity} in that case forbid \emph{strict complementarity}. In particular, for $R=\beta\norm{\freevar}_1 + \delta_{[0, \infty)^n}$ we need to have $q \in (\beta, \ldots, \beta)$, and for $R=\delta_C$ for a convex set $C$, we need to have $q=0$.
Intuitively, this restriction serves to forbid the \emph{finite identification} property \cite{hare2004identifying} of proximal-type methods, as $\{\alpha^n\}$ cannot converge too fast in our techniques for the stability of the inner problem and adjoint with respect to perturbations of $\alpha$.

We now come to our main assumption for the FEFB. It collects conditions related to step lengths, initialization, and the problem functions $F$, $J$, and $R$.
For a constant $c>0$ to be determined by the assumption, we introduce the \term{testing operator}
\begin{equation}
    \label{eq:testing-operator}
    Z \defeq \diag(\phi_u \Id, \Id, \Id).
\end{equation}
The idea, introduced in \cite{tuomov-proxtest} and further explained in \cite{clason2020introduction}, is to \term{test} the algorithm-defining inclusion \eqref{eq:FEFB:implicit} with the linear functional $\iprod{Z\freevar}{\nexxt x-\opt x}$ to obtain a descent estimate with respect to the $ZM$-norm. The operator $Z$ encodes component-specific scalings and convergence rates, although we do not exploit the latter in this manuscript.

\begin{assumption}
    \label{ass:FEFB:main}
    We assume that $U$ is a Hilbert space, $\AlphaSpace$  a separable Hilbert space, and treat the adjoint variable $p\in\linear(U; \AlphaSpace)$ as an element of the inner product space $P$ defined in \eqref{eq:p-space}.
    Let $R: \AlphaSpace \to \extR$ and $J: U \to \R $ be convex, proper, and lower semicontinuous, and assume the same from $F(\freevar, \alpha)\in C^2(U)$ for all $\alpha \in \Dom R.$
    Pick $(\opt u,\opt p,\opt \alpha) \in H^{-1}(0)$ and let $\{(u^m, p^m, \alpha^m)\}_{m\in \mathbb{N}}$ be generated by \cref{alg:FEFB} for a given initial iterate $(u^{0}, p^{0}, \alpha^{0}) \in U \times P \times \Dom R$.
    For a given $r, r_u>0$ we suppose that
    \begin{enumerate}[label=(\roman*)]
        \item\label{ass:FEFB:main:init}
        The relative initialization bound $\norm{u^{1}-S_u(\alpha^{0})} \leq C_u \norm{\alpha^{0} - \opt{\alpha}}$ holds for some $C_u>0$.

        \item\label{ass:FEFB:main:solution-map-and-F}
        There exists in $B(\opt{\alpha}, 2r) \isect \Dom R$ a continuously Fréchet-differentiable and $L_{S_u}$-Lipschitz inner problem solution mapping $S_u: \alpha \mapsto S_u(\alpha) \in \argmin F(\freevar; \alpha)$.
        
        \item\label{ass:FEFB:main:neighbourhoods}
        $F(\opt{u};\freevar)$ is Lipschitz continuously differentiable with factor $L_{\grad F,\opt{u}} > 0$, and $\gamma_F\cdot\Id \le\grad_u^2 F(u; \alpha)\le L_F\cdot\Id$ for all $(u, \alpha) \in B(\opt{u}, r_u) \times ( B(\opt{\alpha}, r) \isect \Dom R)$ for some $\gamma_F, L_F > 0.$
        Moreover, $(u,\alpha) \mapsto\grad_{u}^2 F (u; \alpha)$ and $(u,\alpha) \mapsto\grad_{\alpha u} F (u; \alpha) \in P$ are Lipschitz in $B(\opt{u}, r_u)\times (B(\opt{\alpha}, r) \isect \Dom R)$ with factors $L_{\grad ^2 F}$ and $L_{\grad_{\alpha u} F}$, where we equip $U \times \AlphaSpace$ with the norm $(u, \alpha) \mapsto \norm{u}_U + \norm{\alpha}_{\AlphaSpace}$.

		\item\label{ass:FEFB:main:inner-step-length}
		The inner step length $\tau \in (0, 2\kappa/L_F]$ for some $\kappa \in (0, 1)$.

       \item\label{ass:FEFB:main:outer-objective}
       The outer fitness function $J$ is Lipschitz continuously differentiable with factor $L_{\grad J}$, and $\gamma_\alpha\cdot\Id\le\grad_{\alpha}^2(J\circ S_u)\le L_\alpha\cdot\Id$ in $B(\opt{\alpha}, r)\isect \Dom R$ for some $\gamma_\alpha,L_\alpha>0$.
       Moreover, $R$ is locally prox-$\sigma$-contractive at $\opt\alpha$ for $\opt p\grad_u J(\opt u)$ within $B(\opt{\alpha}, r) \isect \Dom R$ for some $C_R\ge 0$.
    	
       \item \label{ass:FEFB:main:testing-params}
       The constants $\phi_u, C_u > 0$ satisfy
       \begin{gather*}
			\gamma_F (L_{\grad J}N_p + L_{S_p} N_{\grad J})C_u + \frac{L_{\grad F,\opt u}^2}{(1-\kappa)} \phi_u < \gamma_F \gamma_\alpha,
    		\shortintertext{where}
    		\begin{aligned}
    			N_{\grad_{\alpha u} F} & \defeq \max_{\substack{u \in B(\opt{u}, r_u),\\ \alpha \in B(\opt{\alpha}, 2r)\isect \Dom R}}  \onorm{\grad_{\alpha u} F(u,\alpha)},
    			&
    			L_{S_p} & \defeq \gamma_F^{-2} L_{\grad^2 F}N_{\grad_{\alpha u} F} + \gamma_F^{-1} L_{\grad_{\alpha u} F},
    			\\
    			N_{\grad J} & \defeq \max_{\alpha\in B(\opt{\alpha}, r)\isect \Dom R} \norm{\grad_u J (S_u(\alpha))},
    			&
    			N_{\grad S_u} & \defeq
    			\max_{\alpha\in B(\opt{\alpha}, r)\isect \Dom R}
    			\onorm{\grad_{\alpha}S_u(\alpha)},
    			\text{ and}
    			\\
    			N_p & \defeq N_{\grad S_u} + C r \text{ with } C=L_{S_p} C_u.
    			&
    			&
    		\end{aligned}
    	\end{gather*}

        \item\label{ass:FEFB:main:outer-step-length}
        The outer step length $\sigma$ fulfills
        \[
            0 < \sigma \leq \frac{(C_F-1)C_u}{(L_{S_u} +C_FC_u)C_{\alpha}}
            \quad\text{where}\quad
            \begin{cases}
                C_F \defeq \sqrt{1+2\tau\gamma_F(1 - \kappa)},
                \quad\text{and}
                \\
                C_{\alpha} \defeq (N_pL_{\grad J} +  N_{\grad J} L_{S_p}) C_u + L_\alpha + C_R.
            \end{cases}
        \]

        \item\label{ass:FEFB:main:local}
        The initial iterates $u^0$ and $\alpha^0$ are such that the distance-to-solution
        \begin{gather*}
            r_0
            \defeq
            \sqrt{\sigma \phi_u\tau^{-1}\norm{u^{0}-\opt u}^2 + \norm{\alpha^{0} - \opt \alpha}^2}
            =
            \sqrt{\sigma}\norm{x^{0}-\opt x}_{ZM}
        \shortintertext{satisfies}
        	r_0
        	\le
        	r
        	\quad\text{and}\quad
            r_0 \max\{2L_{S_u}, \sqrt{\sigma^{-1}\inv \phi_u\tau}(1+\tau L_{F}) + \tau L_{\grad F,\opt{u}}\}
            \le
            r_u.
        \end{gather*}
\end{enumerate}
\end{assumption}

\begin{remark}[Interpretation]
    \label{rem:FEFB:interpretation}
    Part \cref{ass:FEFB:main:init} of \cref{ass:FEFB:main} ensures that the initial inner problem iterate is good \emph{relative} to the outer problem iterate.
    If $u^1$ solves the inner problem for $\alpha^0$, \cref{ass:FEFB:main:init} holds for any $C_u>0$.
    Therefore, \cref{ass:FEFB:main:init} can always be satisfied by solving the inner problem for $\alpha^0$ to high accuracy. This condition \emph{does not} require $\alpha^0$ to be close to a solution $\opt\alpha$ of the entire problem.
   
    Part \cref{ass:FEFB:main:solution-map-and-F} ensures that the inner problem solution map exists and is well-behaved; we discuss it more in the next \cref{rem:FEFB:solution-map-existence}.
    
    Parts \cref{ass:FEFB:main:neighbourhoods} and \cref{ass:FEFB:main:outer-objective} are second order growth and boundedness conditions, standard in smooth optimization. The nonsmooth $R$ is handled through the prox-$\sigma$-contractivity assumption.
    If $S_u$ is twice Fréchet differentiable, the product and the chain rules establish
	\[
		\iprod{h}{\grad_\alpha^2(J \circ S_u)(\alpha)h}
		=
		\iprod{S_u'(\alpha)h}{\grad_u^2 J(S_u(\alpha))S'(\alpha)h} + \iprod{\grad_u J(S_u(\alpha))}{S_u''(\alpha)(h, h)}
        \quad (h \in \AlphaSpace).
	\]
    If $R=0$, first-order optimality conditions establish $\grad_u J(S_u(\opt\alpha))=0$.
    Therefore, if, further, $J$ is strongly convex and $S_u'(\opt\alpha)$ is invertible, $\gamma\cdot\Id\le\grad_{\alpha}^2(J\circ S_u)(\opt\alpha)$ for some $\gamma>0$. Then additional continuity assumptions establish the positivity required in \cref{ass:FEFB:main:outer-objective} in a neighbourhood of $\opt\alpha$.
    It is also possible to further develop the condition to not depend on the solution mapping at all.

    Dependent on $R$, \cref{ass:FEFB:main:outer-objective} may restrict the outer step length parameter $\sigma$.
    Part \ref{ass:FEFB:main:neighbourhoods} ensures that $u\mapsto \grad_u^2 F(u; \alpha)$ is invertible and $S_p$ is well-defined.
    We will see in \cref{lemma:Iterate_condition_corollary} that the radius $r_u$ is sufficiently large that $\alpha \in B(\opt\alpha, r)$ implies $S_u(\alpha) \in B(\opt u, r_u)$.
    Part \cref{ass:FEFB:main:outer-objective} implies that $\alpha \mapsto \grad_{\alpha}(J\circ S_u)(\alpha)$ is Lipschitz in $B(\opt{\alpha}, r)$.

    Part \cref{ass:FEFB:main:inner-step-length} is a standard step length condition for the inner problem while \ref{ass:FEFB:main:outer-step-length} is a step length condition for the outer problem. It depends on several constants defined in the more technical part \cref{ass:FEFB:main:testing-params}.
    \emph{We can always satisfy the inequality in \cref{ass:FEFB:main:testing-params} by good relative initialisation (small $C_u>0$), as discussed above, and taking the testing parameter $\phi_u$ small.}
    According to the local initialization condition \cref{ass:FEFB:main:local}, the latter can be done if the initialial iterates are close to a solution $(\opt u, \opt \alpha)$ of the entire problem, or if $r_u>0$ can be be taken arbitrarily large.
    If we can take both $r>0$ and $r_u>0$ arbitrarily large, we obtain \emph{global convergence}.
\end{remark}

\begin{remark}[Existence and differentiability of the solution map]
    \label{rem:FEFB:solution-map-existence}
    Suppose $F$ is twice continuously differentiable in both variables, and that $\gamma_F\cdot\Id\le\grad_u^2 F(u; \alpha)$ for all $u\in B(\opt{u}, r_u)$ and $\alpha\in B(\opt{\alpha}, 2r) \isect \Dom R$ for some $\gamma_F>0$.
    Then the implicit function theorem shows the existence of a unique continuously differentiable $S_u$ in a neighborhood of any $\alpha \in B(\opt\alpha,r) \isect \Dom R$.
    Such an $S_u$ is also Lipschitz in a neighborhood of $\alpha$; see, e.g., \cite[Lemma 2.11]{clason2020introduction}.
    If $\AlphaSpace$ is finite-dimensional, a compactness argument gluing together the neighborhoods then proves \cref{ass:FEFB:main}\,\cref{ass:FEFB:main:solution-map-and-F}.
\end{remark}

\subsection{Algorithm: forward--inexact--forward-backward}
\label{subsec:Algorithm:FIFB}

\begin{algorithm}
    \caption{Forward--inexact--forward-backward (FIFB) method }
    \label{alg:FIFB}
    \begin{algorithmic}[1]
        \Require
        Functions $R$, $J$, $F$ as in \cref{ass:FIFB:main}.
        Step length parameters $\tau,\sigma,\theta>0$.
        \State Pick an initial iterate $(u^{0}, p^{0}, \alpha^{0}) \in U\times P \times \AlphaSpace.$
        \For{$k \in \N$}
        \State $u^{k+1} \defeq u^{k} - \tau \grad_{u} F(u^{k}; \alpha^{k})$
        \Comment{inner gradient step}
        \State $p^{k+1} \defeq p^{k} -\theta \left(p^{k} \grad_{u}^2 F(u^{k+1}; \alpha^{k}) + \grad_{\alpha u} F(u^{k+1}; \alpha^{k})\right)$
        \Comment{adjoint step}
        \State  $\alpha^{k+1} \defeq \prox_{ \sigma R}\left (\alpha^{k} - \sigma ( p^{k+1}\grad_{u}J(u^{k+1}))\right )$
        \Comment{outer forward-backward step}
        \EndFor
    \end{algorithmic}
\end{algorithm}

Our second strategy for solving \cref{eq:main_optimality_condition} modifies the first approach to solve the adjoint variable inexactly, so that no costly matrix inversions are required. Instead we perform an update reminiscent of a gradient step. This approach, which we call the FIFB (forward--inexact--forward-backward) reads as \cref{alg:FIFB} and has the implicit form
\begin{equation}
    \label{eq:FIFB:implicit0}
    \begin{cases}
        0 = \tau \grad_{u} F(u^{k}; \alpha^{k}) + u^{k+1} - u^{k}
        \\
        0 = \theta \left(p^{k} \grad_{u}^2 F(u^{k+1}; \alpha^{k}) + \grad_{\alpha u} F(u^{k+1}; \alpha^{k})\right) + p^{k+1} - p^{k}  \\
        0 \in  \sigma(\subdiff R(\nexxt{\alpha}) +   p^{k+1}\grad_{u}J(u^{k+1})) + \alpha^{k+1} - \alpha^{k}.
    \end{cases}
\end{equation}
The implicit form can also be written as \eqref{eq:FEFB:implicit} with
\begin{subequations}
\label{eq:FIFB:implicit}
\begin{equation}
    H_{k+1}(x) \defeq H(x)  + \begin{pmatrix}
        \grad_u F(u^k; \alpha^k) - \grad_u F(u; \alpha)  \\
        p^{k} \grad_{u}^2 F(u; \alpha^{k}) + \grad_{\alpha u}F(u; \alpha^{k}) - \left ( p \grad_{u}^2 F(u; \alpha) + \grad_{\alpha u}F(u; \alpha) \right ) \\
        0
    \end{pmatrix},
\end{equation}
and the preconditioning operator $M\in \mathbb{L}(U\times P \times \AlphaSpace; U\times P \times \AlphaSpace)$,
\begin{equation}
    \label{eq:mkplus1_FIFB}
    M \defeq
    \diag(
        \tau^{-1}\Id,
        \theta^{-1}\Id,
        \sigma^{-1}\Id
    ).
\end{equation}
\end{subequations}

For the testing operator $Z$ we use the structure
\begin{equation}
	\label{eq:testing-operator-FIFB}
	Z \defeq \diag(\phi_u \Id, \phi_p \Id, \Id).
\end{equation}
with the constants $\phi_u, \phi_p>0$ determined in the following assumption. It is the FIFB counterpart of \cref{ass:FEFB:main} for the FEFB, collecting essential structural, step length, and initialization assumptions.

\begin{assumption}\label{ass:FIFB:main}
	We assume that $U$ is a Hilbert space, $\AlphaSpace$  a separable Hilbert space, and treat the adjoint variable $p\in\linear(U; \AlphaSpace)$ as an element of the inner product space $P$ defined in \eqref{eq:p-space}. Let $R: \AlphaSpace \to \extR$ and $J: U \to \R $ be convex, proper, and lower semicontinuous, and assume the same from $F(\freevar, \alpha)$ for all $\alpha \in \Dom R$. Pick $(\opt u,\opt p,\opt \alpha) \in H^{-1}(0)$ and let $\{(u^m, p^m, \alpha^m)\}_{m\in \mathbb{N}}$ be generated by \cref{alg:FIFB} for a given initial iterate $(u^{0}, p^{0}, \alpha^{0}) \in U \times P \times \Dom R$. For given $r, r_u>0$ we suppose that
    \begin{enumerate}[label=(\roman*)]
        \item \label{ass:FIFB:main:init}
        The relative initialization bounds $\norm{u^{1}-S_u(\alpha^{0})} \leq C_u \norm{\alpha^{0} - \opt{\alpha}}$ and $\onorm{p^{1}-\grad_{\alpha}S_u(\alpha^{0})} \leq C_p \norm{\alpha^{0} - \opt{\alpha}}$ hold with some  constants $C_u>0$ and $C_p>0.$

        \item \label{ass:FIFB:main:solution-map-and-F}
        There exists in $B(\opt{\alpha}, 2r) \isect \Dom R$ a continuously Fréchet-differentiable and $L_{S_u}$-Lipschitz inner problem solution mapping $S_u: \alpha \mapsto S_u(\alpha) \in \argmin F(\freevar; \alpha)$.

        \item \label{ass:FIFB:main:neighbourhoods}
        $F(\opt{u};\freevar)$ is Lipschitz continuously differentiable with factor $L_{\grad F,\opt{u}} > 0$, and $\gamma_F\cdot\Id\le\grad_u^2 F(u; \alpha)\le L_F\cdot\Id$ for $u\in B(\opt{u}, r_u)$ and $\alpha\in B(\opt{\alpha}, 2r) \isect \Dom R.$
        Moreover, $(u,\alpha) \mapsto\grad_{u}^2 F (u; \alpha)$ and $(u,\alpha) \mapsto\grad_{\alpha u} F (u; \alpha) \in P$ are Lipschitz in $B(\opt{u}, r_u)\times (B(\opt{\alpha}, r) \isect \Dom R)$ with factors $L_{\grad ^2 F}$ and $L_{\grad_{\alpha u} F}$, where we equip $U \times \AlphaSpace$ with the norm $(u, \alpha) \mapsto \norm{u}_U + \norm{\alpha}_{\AlphaSpace}$.
		
		\item\label{ass:FIFB:main:inner-and-adjoint-step-length}
		The inner step length $\tau \in (0, 2\kappa/L_F]$ for some $\kappa \in (0, 1)$ whereas the adjoint step length $\theta \in (0, 1/L_F)$.

        \item \label{ass:FIFB:main:outer-objective}
        The outer fitness function $J$ is Lipschitz continuously differentiable with factor $L_{\grad J},$ and $\gamma_\alpha\cdot\Id\le\grad_{\alpha}^2(J\circ S_u)\le L_\alpha\cdot\Id$ in $B(\opt{\alpha}, r) \isect \Dom R$ for some $\gamma_\alpha, L_\alpha > 0$.
        Moreover, $R$ is locally prox-$\sigma$-contractive at $\opt\alpha$ for $\opt p\grad_u J(\opt u)$ within $B(\opt{\alpha}, r) \isect \Dom R$ for some $C_R\ge 0$.
        
		\item
		\label{ass:FIFB:main:testing-params}
    	The constants $\phi_u, \phi_p, C_u > 0$ satisfy
    	\begin{gather*}
    		\phi_p
    		\le
    		\phi_u
    		\frac{\gamma_F^2(1-\kappa)}{2 L_F L_{S_p}}
    		\shortintertext{and}
    		L_F L_{S_p} \phi_p + \sqrt{(L_F L_{S_p} \phi_p)^2 + \gamma_F^2 (L_{\grad J}N_p + L_{S_p} N_{\grad J})^2C_u^2} + \frac{L_{\grad F,\opt u}^2}{(1-\kappa)} \phi_u < \gamma_F \gamma_\alpha,
        \end{gather*}
        where
        \[
    		\begin{aligned}
    			N_{\grad_{\alpha u} F} & \defeq \max_{\substack{u \in B(\opt{u}, r_u),\\ \alpha \in B(\opt{\alpha}, 2r)\isect \Dom R}}  \onorm{\grad_{\alpha u} F(u,\alpha)},
    			&
    			L_{S_p} & \defeq \gamma_F^{-2} L_{\grad^2 F}N_{\grad_{\alpha u} F} + \gamma_F^{-1} L_{\grad_{\alpha u} F},
    			\\
    			N_{\grad J} & \defeq \max_{\alpha\in B(\opt{\alpha}, r)\isect \Dom R} \norm{\grad_u J (S_u(\alpha))},
    			&
    			N_{\grad S_u} & \defeq
    			\max_{\alpha\in B(\opt{\alpha}, r)\isect \Dom R}
    			\onorm{\grad_{\alpha}S_u(\alpha)},
    			\text{ and}
    			\\
    			N_p & \defeq N_{\grad S_u} + C r \text{ with } C=L_{S_p} C_u.
    			&
    			&
    		\end{aligned}
    	\]
        \item \label{ass:FIFB:main:outer-step-length}
        The outer step length $\sigma$ satisfies
        \[
            0 < \sigma \leq \frac{1}{C_\alpha}\min\left\{\frac{(C_F-1)C_u}{L_{S_u} +C_FC_u}, \frac{(C_{F,S}-1)C_p- (1+C_{F,S})L_{S_p} C_u}{(1+L_{S_u})L_{S_p}+C_{F,S}C_p- (1+C_{F,S})L_{S_p}C_u} \right\}
        \]
        with
        \[
            \begin{aligned}
                C_F & \defeq \sqrt{1+2\tau\gamma_F(1 - \kappa)},
                \qquad
                C_{F,S} \defeq \sqrt{(1+\theta \gamma_F)/(1-\theta \gamma_F)}
                \quad\text{and}
                \\
                C_{\alpha} & \defeq N_pL_{\grad J}C_u +  N_{\grad J}\max \{C_p, L_{S_p} C_u\} + L_\alpha + C_R.
            \end{aligned}
        \]

        \item \label{ass:FIFB:main:local}
        The initial iterate $(u^0, p^0, \alpha^0)$ is such that the distance-to-solution
        \begin{gather*}
        	r_0
        	\defeq
        	\sqrt{\sigma \phi_u\tau^{-1}\norm{u^{0}-\opt u}^2 + \sigma \phi_p\theta^{-1}\onorm{p^{0}-\opt p}^2 + \norm{\alpha^{0} - \opt \alpha}^2}
        	=
        	\sqrt{\sigma}\norm{x^{0}-\opt x}_{ZM}
        	\shortintertext{satisfies}
        	r_0
        	\le
        	r
        	\quad\text{and}\quad
        	r_0 \max\{2(C_u + L_{S_u}), \sqrt{\sigma^{-1}\inv \phi_u\tau}(1+\tau L_{F}) + \tau L_{\grad F,\opt{u}}\}
        	\le
        	r_u.
        \end{gather*}
\end{enumerate}
\end{assumption}

\begin{remark}[Interpretation]
    The interpretation of \cref{ass:FEFB:main} in \cref{rem:FEFB:solution-map-existence} also applies to \cref{ass:FIFB:main}.
    We stress that \emph{to satisfy the inequality in \cref{ass:FIFB:main:testing-params}, it suffices to ensure small $C_u>0$ by good relative initialization of $u$ and $p$ with respect to $\alpha$, and choosing the testing parameters $\phi_u, \phi_p>0$ small enough.}
    According to \cref{ass:FIFB:main:local}, the latter can be done by initializing close to a solution, or if the radii $r_u>0$ is large.
\end{remark}

\section{Convergence analysis}
\label{sec:convergence}

We now prove the convergence of the FEFB (\cref{alg:FEFB}) and the FIFB (\cref{alg:FIFB}) in the respective \cref{subsec:convergence:FEFB,subsec:convergence:FIFB}.
Before this we start with common results.
Our proofs are self-contained, but follow on the “testing” approach of \cite{tuomov-proxtest} (see also \cite{clason2020introduction}).
The main idea is to prove a monotonicity-type estimate for the operator $H_{k+1}$ occurring in the implicit forms \cref{eq:FEFB:implicit,eq:FIFB:implicit} of the algorithms, and then use the three-point identity \eqref{3-point-identity} with respect to $ZM$-norms and inner products.
This yields an inequality that readily yields an estimate from which convergence rates can be observed.
The main results for the FEFB and the FIFB and in the respective \cref{thr:grad-exact-grad_convergence,thr:grad-grad-grad_convergence}.

Throughout, we assume that either \cref{ass:FEFB:main} (FEFB) or \ref{ass:FIFB:main} (FIFB) holds, and \emph{tacitly use the constants from the relevant one}.
We also tacitly take it that $\alpha^k \in \Dom R$ for all $k \in \N$, as this is guaranteed by the assumptions for $k=0$, and by the proximal step in the algorithms for $k \ge 1$.

\subsection{General results}\label{subsec:convergence:general}

Our main goal here is to bound the error in the inner and adjoint iterates $u^k$ and $p^k$ in terms of the outer iterates $\alpha^k$. We also derive bounds on the outer steps, and local monotonicity estimates.
We first show that the solution mapping for the adjoint equation \eqref{eq:p-row-oc} is Lipschitz.

\begin{lemma}\label{lemma:S_p-Lipschitz}
    Suppose $(u, \alpha) \mapsto \grad_{u}^2 F (u; \alpha)$ and $(u, \alpha) \mapsto\grad_{\alpha u} F (u; \alpha) \in P$ are Lipschitz continuous with the respective constants $L_{\grad ^2 F}$ and $L_{\grad_{\alpha u} F}$ in some bounded closed set $V_u \times V_{\alpha}.$  Also assume that $\gamma_F\cdot \Id \leq \grad_u^2 F(u; \alpha)$ and $\onorm{\grad_{\alpha u} F} \leq N_{\grad_{\alpha u} F}$ in $V_u \times V_{\alpha}$ for some $\gamma_F, N_{\grad_{\alpha u}}>0$.
    Then $S_p$ is Lipschitz continuous in $V_u \times V_{\alpha},$ i.e.
    \[
        \onorm{S_p(u_1, \alpha_1) - S_p(u_2, \alpha_2)} \leq L_{S_p}(\norm{u_1 - u_2}+\norm{\alpha_1-\alpha_2})
    \]
    for $u_1,u_2\in V_u$ and $\alpha_1,\alpha_2 \in V_{\alpha}$ with factor $L_{S_p} \defeq \gamma_F^{-2} L_{\grad^2 F}N_{\grad_{\alpha u} F} + \gamma_F^{-1} L_{\grad_{\alpha u} F}.$
\end{lemma}
\begin{proof}
    Using the definition of $S_p$ in \eqref{def:S_p}, we rearrange
    \[
        \begin{aligned}[t]
            S_p(u_1, \alpha_1) - S_p(u_2, \alpha_2)
            &
            =
            \left( \grad_{\alpha u} F (u_1; \alpha_1) - \grad_{\alpha u} F (u_2; \alpha_2)\right)(\grad_u^2 F(u_1; \alpha_1))^{-1}
            \\
            \MoveEqLeft[-1]
            + \grad_{\alpha u} F (u_2; \alpha_2)\left( (\grad_u^2 F(u_1; \alpha_1))^{-1} - 	(\grad_u^2 F(u_2; \alpha_2))^{-1}\right) .
        \end{aligned}
    \]
    Thus the triangle inequality and the operator norm inequality \cref{thm:separable:properties}\,\ref{item:separable:operator-norm-combo} give
    \begin{equation}
        \label{S_p-ineq1}
        \begin{aligned}[t]
            \MoveEqLeft[1]\onorm{S_p(u_1, \alpha_1) - S_p(u_2, \alpha_2)}
            \\
            &
            \leq
            \norm{(\grad_u^2 F(u_1; \alpha_1))^{-1}}\onorm{ \grad_{\alpha u} F (u_1; \alpha_1) - \grad_{\alpha u} F (u_2; \alpha_2)}
            \\
            \MoveEqLeft[-1]
            + \onorm{ \grad_{\alpha u} F (u_2; \alpha_2)}\norm{(\grad_u^2 F(u_1; \alpha_1))^{-1} - (\grad_u^2 F(u_2; \alpha_2))^{-1}}
            =: E_1 + E_2.
        \end{aligned}
    \end{equation}
    The assumption $\gamma_F\cdot \Id \leq \grad_u^2 F(u; \alpha)$ implies $\norm{(\grad_u^2 F(u; \alpha))^{-1}}\leq \gamma_F^{-1}.$
    Therefore, also using the Lipschitz continuity of $(u, \alpha) \mapsto \grad_{\alpha u} F (u; \alpha)$ in $V_u\times V_{\alpha},$ we get
    \begin{equation}
        \label{S_p-ineq2}
        E_1 \leq \gamma_F^{-1} L_{\grad_{\alpha u} F}\left(\norm{u_1-u_2} + \norm{\alpha_1 - \alpha_2}\right).
    \end{equation}
    Towards estimating the second term on the right hand side of \eqref{S_p-ineq1}, we observe that
    \[
        A^{-1} - B^{-1}= A^{-1}B B^{-1} - A^{-1}A B^{-1} = A^{-1}(A-B)B^{-1}
    \]
    for any invertible linear operators $A, B$.
    Then we use $\onorm{\grad_{\alpha u} F} \leq N_{\grad_{\alpha u} F}$ and the Lipschitz continuity of $\grad_{u}^2 F (u; \alpha)$ to obtain
    \[%
        \begin{aligned}[t]
            E_2
            &
            =
            \norm{\grad_u^2 F(u_1;\alpha_1)^{-1}(\grad_u^2 F(u_1;\alpha_1)- \grad_u^2 F(u_2;\alpha_2))\grad_u^2 F(u_2;\alpha_2)^{-1}}
            \cdot \onorm{ \grad_{\alpha u} F (u_2; \alpha_2)}
            \\
            &
            \leq
            N_{\grad_{\alpha u} F}
            \norm{\grad_u^2 F(u_1;\alpha_1)^{-1}}\norm{\grad_u^2 F(u_2;\alpha_2)^{-1}}\norm{(\grad_u^2 F(u_1;\alpha_1)- \grad_u^2 F(u_2;\alpha_2))}
            \\
            &
            \leq
            \gamma_F^{-2} L_{\grad^2 F}N_{\grad_{\alpha u} F}\left(\norm{u_1-u_2} + \norm{\alpha_1 - \alpha_2}\right).
        \end{aligned}
    \]
    Inserting this inequality and \eqref{S_p-ineq2} into \eqref{S_p-ineq1} establishes the claim.
\end{proof}

We now prove two simple step length bounds.

\begin{lemma}
    \label{lemma:sigma_bound_corollary}
    Let \cref{ass:FEFB:main} or \ref{ass:FIFB:main} hold. Then $\sigma < 1/L_\alpha$ and $1 < C_F <  \sqrt{1+\gamma_F/L_F}$.
\end{lemma}

\begin{proof}
    We have $C_F>1$ since $\kappa<1$ forces $2\tau\gamma_F(1 - \kappa)>0.$
    \Cref{ass:FEFB:main}\,\ref{ass:FEFB:main:inner-step-length} or \ref{ass:FIFB:main}\,\ref{ass:FIFB:main:inner-and-adjoint-step-length} implies $2\tau\gamma_F(1 - \kappa)<4\gamma_F(\kappa - \kappa^2)/L_F \leq\gamma_F/L_F.$ Therefore $ C_F <  \sqrt{1+ \gamma_F/L_F}.$
    For $C_F,C_u, L_{S_u}>0$ it holds $C_FC_u-C_u< L_{S_u} +C_FC_u.$
    Hence \cref{ass:FEFB:main}\,\ref{ass:FEFB:main:outer-step-length} or \ref{ass:FIFB:main}\,\ref{ass:FIFB:main:outer-step-length} gives
    \[
        \sigma \leq \frac{(C_F-1)C_u}{C_{\alpha}(L_{S_u} +C_FC_u)} < \frac{1}{C_{\alpha}} = \frac{1}{C_u (L_{S_p}  N_{\grad J} + N_{p}) + L_\alpha + C_R} < \frac{1}{L_\alpha} .
        \qedhere
    \]
\end{proof}

The next lemma explains the latter inequality for $r_0$ in \cref{ass:FEFB:main}\,\cref{ass:FEFB:main:local} and \ref{ass:FIFB:main}\,\cref{ass:FIFB:main:local}.
For $u^n$ and $\alpha^n$ close enough to the respective solutions, it bounds the next iterate $u^{n+1}$ and the true inner problem solution $S_u(\alpha^n)$ for $\alpha^n$ to the $r_u$-neighborhood of $\opt u$.

\begin{lemma}
    \label{lemma:Iterate_condition_corollary}
    Suppose \cref{ass:FEFB:main} or \ref{ass:FIFB:main} hold and $\alpha^{n}\in B(\opt{\alpha}, r_0)$, as well as $u^{n}\in B(\opt{u}, \sqrt{\sigma^{-1}\inv \phi_u\tau}r_0)$.
    Then $u^{n+1}\in B(\opt{u}, r_u)$ and  $S_u(\alpha^n)\in B(\opt{u}, r_u).$
\end{lemma}

\begin{proof}
    The inner gradient step of \cref{alg:FEFB} or \ref{alg:FIFB} with $u^{n}\in B(\opt{u}, \sqrt{\sigma^{-1}\inv \phi_u\tau}r_0)$ give
    \[
        \norm{u^{n+1} - \opt{u}} \le \norm{u^{n+1} - u^{n}} + \norm{u^{n} - \opt{u}} \leq \tau\norm{\grad_u F(u^n; \alpha^n)} + \sqrt{\sigma^{-1}\inv \phi_u\tau} r_0.
    \]
    Using $\grad_u F(\opt{u}; \opt{\alpha})=0$, $\alpha^{n}\in B(\opt{\alpha}, r_0)$, and the Lipschitz continuity of $F(\opt{u};\freevar)$ and $F(\freevar;\alpha^n)$ from \cref{ass:FEFB:main}\,\ref{ass:FEFB:main:neighbourhoods} or \ref{ass:FIFB:main}\,\ref{ass:FIFB:main:neighbourhoods} we continue to estimate, as required
    \begin{align*}
        \norm{u^{n+1} - \opt{u}}
        &\leq \tau\norm{\grad_u F(u^n; \alpha^n) - \grad_u F(\opt{u}; \alpha^n) + \grad_u F(\opt{u}; \alpha^n) - \grad_u F(\opt{u}; \opt{\alpha})} + \sqrt{\sigma^{-1}\inv \phi_u\tau} r_0 \\
        &\leq \tau (L_F\norm{u^{n} - \opt{u}} + L_{\grad F, \opt{u}}\norm{\alpha^{n} - \opt{\alpha}}) + \sqrt{\sigma^{-1}\inv \phi_u\tau} r_0  \\
        &\leq (\sqrt{\sigma^{-1}\inv \phi_u\tau}(1+\tau L_F) + \tau L_{\grad F,\opt{u}})r_0 \le r_u.
    \end{align*}
    
    Next, the Lipschitz continuity of $S_u$ in $B(\opt{\alpha}, 2r)$ from  \cref{ass:FEFB:main}\,\cref{ass:FEFB:main:solution-map-and-F} or \ref{ass:FIFB:main}\,\cref{ass:FIFB:main:solution-map-and-F} with $\alpha^n\in B(\opt{\alpha}, r_0)$ and $r_0\le r$ from \cref{ass:FEFB:main}\,\cref{ass:FEFB:main:local} or \ref{ass:FIFB:main}\,\cref{ass:FIFB:main:local} imply
    \[
        \norm{S_u(\alpha^{n}) - \opt{u}} = \norm{S_u(\alpha^{n}) - S_u(\opt{\alpha})} \le L_{S_u}\norm{\alpha^{n} - \opt{\alpha}} \le L_{S_u}r_0 \le r_u.
        \qedhere
    \]
\end{proof}

We now introduce a working condition that we later prove. It guarantees that the Lipschitz and Hessian properties of \cref{ass:FEFB:main}\cref{ass:FEFB:main:neighbourhoods,ass:FEFB:main:outer-objective,ass:FEFB:main:solution-map-and-F}, or \cref{ass:FIFB:main}\cref{ass:FIFB:main:neighbourhoods,ass:FIFB:main:outer-objective,ass:FIFB:main:solution-map-and-F} hold at iterates.

\begin{assumption}[Iterate locality]
    \label{Iterate_condition}
    Let $r_0\le r$ and $N_p$ be defined in either \cref{ass:FEFB:main} or \ref{ass:FIFB:main}.
    Then this assumption holds for a given $n \in \N$ if
    \[
        \alpha^{n}\in B(\opt{\alpha}, r_0),\quad
        u^{n}\in B(\opt{u}, \sqrt{\sigma^{-1}\inv \phi_u\tau}r_0),
        \quad \text{ and }\quad
        \onorm{p^{n+1}} \leq N_p.
    \]
\end{assumption}

Indeed, the next two lemmas show that if \cref{Iterate_condition} holds for $n=k$ along with some further conditions, then it holds for $n=k+1$.

\begin{lemma}
    \label{lemma:p-np}
    Suppose either \cref{ass:FEFB:main} or \ref{ass:FIFB:main} holds.
    Let $n \in \N$ and suppose
    \begin{equation}
        \label{grad:adjoint_closeness_formula}
        \onorm{p^{n+1}-\grad_{\alpha}S_u(\alpha^n)} \leq  C\norm{\alpha^n- \opt \alpha}
    \end{equation}
    with $\alpha^n \in B(\opt{\alpha}, r)$.
    Then $\onorm{p^{n+1}} \leq N_p$.
\end{lemma}
\begin{proof}
    We estimate using \eqref{grad:adjoint_closeness_formula} and the definitions of the relevant constants in \cref{ass:FEFB:main} or \ref{ass:FIFB:main} that
    \begin{align*}
        \onorm{p^{n+1}}
        &
        \leq
        \onorm{\grad_{\alpha}S_u(\alpha^n)} + \onorm{p^{n+1}-\grad_{\alpha}S_u(\alpha^n)}
        \\
        &
        \leq
        N_{\grad S_u} + C \norm{\alpha^n- \opt \alpha} \leq N_{\grad S_u} + C r = N_p.
        \qedhere
    \end{align*}
\end{proof}

\begin{lemma}
    \label{lemma:assumptions}
    Let $k \in \N$.
    Suppose either \cref{ass:FEFB:main} or \ref{ass:FIFB:main} holds; \cref{Iterate_condition} holds for $n=k$; and that \eqref{grad:adjoint_closeness_formula} holds for $n=k+1$.
    If also $\norm{x^{n+1} - \opt{x}}_{ZM} \leq \norm{x^n - \opt{x}}_{ZM}$
    for $n\in\{0,\ldots,k\}$, then \cref{Iterate_condition} holds for $n=k+1$.
\end{lemma}

\begin{proof}
    Summing $\norm{x^{n+1} - \opt{x}}_{ZM} \leq \norm{x^n - \opt{x}}_{ZM}$ over $n=0,\ldots,k$ gives $\norm{x^{k+1} - \opt{x}}_{ZM} \leq  \norm{x^{0} - \opt{x}}_{ZM} = \sigma^{-1/2}r_0.$ By the definitions of $Z$ and  $M$ in \eqref{eq:testing-operator} or \eqref{eq:testing-operator-FIFB}, and  \eqref{eq:mkplus1} or \eqref{eq:mkplus1_FIFB} respectively, it follows $\alpha^{k+1} \in B(\opt \alpha, r_0)$ and $u^{k+1}\in B(\opt{u}, \sqrt{\sigma^{-1}\inv \phi_u\tau}r_0)$ as required.
    We finish by using \cref{lemma:p-np} with $n=k+1$ to establish  $\onorm{p^{k+2}} \leq N_p$.
\end{proof}

We next prove a monotonicity-type estimate for the inner objective.
For this we need need the following three-point monotonicity inequality.

\begin{theorem}
    \label{thm:3point-monotonicity-strongly-convex}
    Let $z,\opt{x}\in X.$ Suppose $F\in C^2(X),$  and for some $L>0$ and $\gamma \ge 0$ that $\gamma \cdot \Id \le \grad^2 F(\zeta) \le L \cdot \Id$ for all $\zeta \in [\opt x, z] \defeq \{\opt{x} + s(z - \opt{x}) \mid s\in[0,1] \}$. Then
    \begin{equation}\label{ineq:3point-monotonicity-strongly-convex}
        \iprod{\grad F(z)- \grad F(\opt{x})}{x-\opt{x}} \geq \gamma (1 - \beta)\norm{x-\opt{x}}^2 - \frac{L}{4\beta}\norm{x-z}^2
        \quad (\beta \in (0, 1],\, x \in X).
    \end{equation}
\end{theorem}
\begin{proof}
    The proof follows that of \cite[Lemma 15.1]{clason2020introduction} whose statement unnecessarily takes $\zeta$ in neighborhood of $\opt x$ instead of just the interval $[\opt x, z]$.
\end{proof}

\begin{lemma}\label{lemma:inner-problem-monotonicity}
    Let $n \in \N$.
    Suppose either \cref{ass:FEFB:main} or \ref{ass:FIFB:main}, and \ref{Iterate_condition} hold.
    Then for any $\kappa\in (0,1)$, we have
    \begin{equation}
        \label{u_row_Young:general}
        \begin{aligned}[t]
        \iprod{\grad_{u}F(u^{n}; \alpha^{n})}{u^{n+1}-\opt u}
        &
        \ge
        \frac{\gamma_F(1 - \kappa)}{2}\norm{u^{n+1}-\opt u}^2
        \\
        \MoveEqLeft[-1]
        - \frac{L_F}{4\kappa}\norm{u^{n+1}-u^{n}}^2
        - \frac{L_{\grad F, \opt{u}}^2}{2\gamma_F(1 - \kappa)}\norm{\alpha^{n} - \opt \alpha}^2.
        \end{aligned}
    \end{equation}
\end{lemma}

\begin{proof}
    \Cref{ass:FEFB:main}\,\ref{ass:FEFB:main:neighbourhoods} or \ref{ass:FIFB:main}\,\ref{ass:FIFB:main:neighbourhoods} with $\alpha^n\in B(\opt{\alpha}, r)$ and
    $u^{n}\in B(\opt{u},r_u)$ from \cref{Iterate_condition} give $\gamma_F \cdot \Id \le \grad_u^2 F(u; \alpha^n) \le L_F\cdot \Id$ for all $u \in [\opt u, u^n]$.
    We have $\grad_{u}F(\opt{u}; \opt{\alpha})=0$ since $0\in H(\opt{u}, \opt{p}, \opt{\alpha})$.
    Therefore \cref{thm:3point-monotonicity-strongly-convex} yields
    \begin{multline*}
        \iprod{\grad_{u}F(u^{n}; \alpha^{n})}{u^{n+1}-\opt u}
        =
        \iprod{\grad_{u}F(u^{n}; \alpha^{n}) - \grad_{u}F(\opt{u}; \alpha^{n}) + \grad_{u}F(\opt{u}; \alpha^{n}) -\grad_{u}F(\opt{u}; \opt{\alpha}) }{u^{n+1}-\opt u}
        \\
        \ge
        \gamma_F(1 - \kappa)\norm{u^{n+1}-\opt u}^2
        - \frac{L_F}{4\kappa}\norm{u^{n+1}-u^{n}}^2
        - \abs{\iprod{\grad_{u}F(\opt{u}; \alpha^{n}) -\grad_{u}F(\opt{u}; \opt{\alpha})}{u^{n+1}-\opt u}}.
    \end{multline*}
    Young's inequality and the definition of $L_{\grad F, \opt{u}}$ in \cref{ass:FEFB:main}\,\ref{ass:FEFB:main:neighbourhoods} or \ref{ass:FIFB:main}\,\ref{ass:FIFB:main:neighbourhoods} now readily establishes the claim.
\end{proof}

The next lemma bounds the steps taken for the outer problem variable.

\begin{lemma}
    \label{lemma:alpha-option}
    Let $n \in \N$.
    Suppose either \cref{ass:FEFB:main} or \ref{ass:FIFB:main} hold, as do \cref{Iterate_condition}, \cref{grad:adjoint_closeness_formula}, and
    \begin{equation}
    	\label{grad:u_u_alpha_closeness_formula}
    	\norm{u^{n+1}-S_u(\alpha^{n})} \leq C_u 		\norm{\alpha^{n} - \opt{\alpha}}.
    \end{equation}
    Then
    \begin{gather}
        \label{ineq:alpha_in_exact_grad}
        \norm{\alpha^{n+1}-\alpha^n}
        \le \sigma [(N_p L_{\grad J} C_u + N_{\grad J} C + L_\alpha) + C_R]\norm{\alpha^{n} - \opt{\alpha}}
    \shortintertext{and}
        \label{ineq:S_u_exactness_essential}
        C_u\norm{\alpha^{n} - \opt{\alpha}}+ L_{S_u}\norm{\alpha^{n+1}-\alpha^n}
        \le
        C_F C_u \bigl(\norm{\alpha^n - \opt{\alpha}} - \norm{\alpha^{n+1} - \alpha^n}\bigr)
        \le
        C_F C_u \norm{\alpha^{n+1}-\opt{\alpha}}.
    \end{gather}
\end{lemma}

\begin{proof}
    Using the $\alpha$-update of \cref{alg:FEFB} or \ref{alg:FIFB}, we estimate
    \[
        \begin{aligned}[t]
            \norm{\alpha^{n+1}-\alpha^n}
            &
            =
            \norm{[\prox_{\sigma R}(\alpha^n - \sigma p^{n+1}\grad_u J(u^{n+1}))- \alpha^n]-[\prox_{\sigma R}(\opt\alpha - \sigma \opt p\grad_u J(\opt u))-\opt \alpha]}
            \\
            &
            \le
            \norm{[\prox_{\sigma R}(\alpha^n - \sigma p^{n+1}\grad_u J(u^{n+1}))- \alpha^n]-[\prox_{\sigma R}(\alpha^n - \sigma \opt p\grad_u J(\opt u))-\alpha^n]}
            \\
            \MoveEqLeft[-1]
            +
            \norm{[\prox_{\sigma R}(\alpha^n - \sigma \opt p\grad_u J(\opt u))- \alpha^n]-[\prox_{\sigma R}(\opt\alpha - \sigma \opt p\grad_u J(\opt u))-\opt \alpha]}.
        \end{aligned}
    \]
    Since proximal maps are $1$-Lipschitz, and $R$ is by \cref{ass:FEFB:main}\,\ref{ass:FEFB:main:outer-objective} or \ref{ass:FIFB:main}\,\ref{ass:FIFB:main:outer-objective} locally prox-$\sigma$-contractive at $\opt\alpha$ for $\opt p\grad_u J(\opt u)$ within $B(\opt{\alpha}, r) \isect \Dom R$ with factor $C_R$, it follows
    \begin{equation}
        \label{eq:prox-est}
        \norm{\alpha^{n+1}-\alpha^n}
        \le
        \sigma \norm{p^{n+1}\grad_u J(u^{n+1})- \opt p\grad_u J(\opt u)}
        + \sigma C_R \norm{\alpha^n-\opt\alpha}
        =: \sigma Q + \sigma C_R \norm{\alpha^n-\opt\alpha}.
    \end{equation}

    We have $\opt p\grad_u J(\opt u)=\grad_\alpha S_u(\opt\alpha) \grad_u J(S_u(\opt\alpha))=\grad_\alpha(J \circ S_u)(\opt\alpha)$, where $\grad_{\alpha}(J\circ S_u)$ is $L_\alpha$-Lipschitz in $B(\opt{\alpha}, r)\ni \alpha^n$ by \cref{ass:FEFB:main}\,\ref{ass:FEFB:main:outer-objective} or \ref{ass:FIFB:main}\,\ref{ass:FIFB:main:outer-objective}. Hence
    \[
        \begin{aligned}[t]
            Q
            &
            \le
            \norm{p^{n+1}\grad_u J(u^{n+1}) - \grad_{\alpha}(J\circ S_u)(\alpha^n) + \grad_{\alpha}(J\circ S_u)(\alpha^n) -\opt{p}\,\grad_u J(\opt{u})}
            \\
            &
            \le
            \norm{p^{n+1}\grad_u J(u^{n+1}) - \grad_{\alpha}S_u(\alpha^n)\grad_u J(S_u(\alpha^n))} + L_\alpha\norm{\alpha^n - \opt{\alpha}}.
        \end{aligned}
    \]
    Using the Lipschitz continuity of $\grad_u J$ from \cref{ass:FEFB:main}\,\ref{ass:FEFB:main:outer-objective} or \ref{ass:FIFB:main}\,\ref{ass:FIFB:main:outer-objective}, we continue
    \[
        \begin{aligned}[t]
            Q
            &
            \le
            \norm{p^{n+1}(\grad_u J(u^{n+1})-\grad_u J(S_u(\alpha^n)) + (p^{n+1}-\grad_{\alpha}S_u(\alpha^n))\grad_u J(S_u(\alpha^n)) } + L_\alpha \norm{\alpha^n - \opt{\alpha}}
            \\
            &
            \le
            \onorm{p^{n+1}}L_{\grad J}\norm{u^{n+1}-S_u(\alpha^n)} + \onorm{p^{n+1}-\grad_{\alpha}S_u(\alpha^n)}\norm{\grad_u J(S_u(\alpha^n))} + L_\alpha \norm{\alpha^n - \opt{\alpha}}.
        \end{aligned}
    \]
    We have  $\onorm{p^{n+1}}\leq N_{p}$ and $\alpha^n\in B(\opt{\alpha}, r)$ by \cref{Iterate_condition}. Hence $\norm{\grad_u J(S_u(\alpha^n))} \leq N_{\grad J}$ by the definition in \cref{ass:FEFB:main}\,\cref{ass:FEFB:main:testing-params} or \ref{ass:FIFB:main}\,\cref{ass:FIFB:main:testing-params}.
    Using \cref{grad:adjoint_closeness_formula} and \cref{grad:u_u_alpha_closeness_formula} therefore give
    \[
        Q
        \le
        N_p L_{\grad J} C_u \norm{\alpha^n - \opt{\alpha}} + N_{\grad J} C \norm{\alpha^n - \opt{\alpha}} + L_\alpha \norm{\alpha^n - \opt{\alpha}}
        = (C_{\alpha}-C_R)\norm{\alpha^n - \opt{\alpha}}.
    \]
    Inserting this into \eqref{eq:prox-est}, we obtain \eqref{ineq:alpha_in_exact_grad}.
    \Cref{ass:FEFB:main}\,\cref{ass:FEFB:main:outer-step-length}  or \cref{ass:FIFB:main}\,\cref{ass:FIFB:main:outer-step-length} and \eqref{ineq:alpha_in_exact_grad} then yield
    \[
        (L_{S_u} +C_FC_u)\norm{\alpha^{n+1} - \alpha^n}
        \le
        \sigma (L_{S_u} +C_FC_u)C_{\alpha}
        \norm{\alpha^n - \opt{\alpha}}
        \le
        (C_F-1)C_u \norm{\alpha^n - \opt{\alpha}}.
    \]
    Rearranging terms and finishing with the triangle inequality we get \eqref{ineq:S_u_exactness_essential}.
\end{proof}

\begin{remark}[Gradient steps with respect to $R$]
    We could (in both FEFB and FIFB) also take a gradient step instead of a proximal step with respect to $R$ with $L_{\grad R}$-Lipschitz gradient. That is, we would perform for the outer problem the update
    \[
        \alpha^{n+1} = \alpha^n - \sigma[p^{n+1}\grad_u J(u^{n+1}) + \grad R(\alpha^n)].
    \]
    This can be shown to be convergent by changing \eqref{eq:prox-est} to
    \[
        \begin{aligned}[t]
            \norm{\alpha^{n+1}-\alpha^n}
            &
            =
            \sigma\norm{p^{n+1}\grad_u J(u^{n+1}) + \grad R(\alpha^n)}
            \\
            &
            =
            \sigma\norm{p^{n+1}\grad_u J(u^{n+1}) - \opt p\,\grad_u J(\opt u) + \grad R(\alpha^n)-\grad R(\opt \alpha)}
            \\
            &
            \le
            \sigma \bigl(
                \norm{p^{n+1}\grad_u J(u^{n+1})- \opt p\,\grad_u J(\opt u)}
            + L_{\grad R}\norm{\alpha^n-\opt\alpha}
            \bigr).
        \end{aligned}
    \]
\end{remark}

We next prove that if an inner problem iterate has small error, and we take a short step in the outer problem, then also the next inner problem iterate has small error.

\begin{lemma}\label{lemma:grad:u_u_alpha_closeness_formula}
    Let $k \in \N$.
    Suppose \cref{ass:FEFB:main} or \ref{ass:FIFB:main} hold.
    If \cref{Iterate_condition}, \cref{grad:adjoint_closeness_formula}, and \eqref{grad:u_u_alpha_closeness_formula} hold for $n=k,$
    then
    \eqref{grad:u_u_alpha_closeness_formula} holds for $n=k+1$ and we have $\alpha^{k+1}\in B(\opt\alpha, 2r_0)$.
\end{lemma}

\begin{proof}
    We plan to use \cref{thm:3point-monotonicity-strongly-convex} on $F(\freevar; \alpha^{k+1})$ followed by the three-point identity and simple manipulations.
    We begin by proving the conditions of the theorem.

    First, we show that both $u^{k+1}\in B(\opt{u},r_u)$ and $S_u(\alpha^{k+1})\in B(\opt{u},r_u)$.
    The former is immediate from \cref{Iterate_condition} and \cref{lemma:Iterate_condition_corollary}.
    For the latter we use \eqref{ineq:S_u_exactness_essential} of \cref{lemma:alpha-option}. Its first inequality readily implies either $\norm{\alpha^{k} - \opt{\alpha}} > \norm{\alpha^{k+1} - \alpha^{k}}$ or $\alpha^{k+1} = \opt\alpha$.
    In the latter case $S_u(\alpha^{k+1})=\opt u \in B(\opt{u},r_u)$.
    In the former, using $\alpha^k\in B(\opt{\alpha},r_0)$, we get
    \[
        \norm{\alpha^{k+1} - \opt{\alpha}} \leq \norm{\alpha^{k+1} - \alpha^{k}} + \norm{\alpha^{k}- \opt{\alpha}} < 2 \norm{\alpha^{k}- \opt{\alpha}} \leq 2r_0
    \]
    Therefore we can use the Lipschitz continuity of $S_u$ in $B(\opt{\alpha},2r)$ from \cref{ass:FEFB:main}\, \ref{ass:FEFB:main:solution-map-and-F} or \ref{ass:FIFB:main}\,\ref{ass:FIFB:main:solution-map-and-F} to estimate
    \[
        \norm{S_u(\alpha^{k+1}) - \opt{u}} = \norm{S_u(\alpha^{k+1}) - S_u(\opt{\alpha})} \leq  L_{S_u}\norm{\alpha^{k+1} - \opt{\alpha}} \leq L_{S_u}2r_0.
    \]
    This implies $S_u(\alpha^{k+1})\in B(\opt{u},r_u)$ by \cref{ass:FEFB:main}\,\ref{ass:FEFB:main:local} or \ref{ass:FIFB:main}\,\ref{ass:FIFB:main:local}.

    Since both $u^{k+1}, S_u(\alpha^{k+1}) \in B(\opt{u},r_u)$, \cref{ass:FEFB:main}\,\ref{ass:FEFB:main:neighbourhoods} or \ref{ass:FIFB:main}\,\ref{ass:FIFB:main:neighbourhoods}
    shows that $\gamma_F \cdot \Id \le \grad^2 F(u) \le L_F\cdot \Id$ for $u \in [S_u(\alpha^{k+1}), u^{k+1}]$.
    Consequently \cref{thm:3point-monotonicity-strongly-convex} and $\grad_{u}F(S_u(\alpha^{k+1}); \alpha^{k+1}) = 0$ give
    \[
        \iprod{\grad_{u}F(u^{k+1}; \alpha^{k+1})}{u^{k+2}-S_u(\alpha^{k+1})}
        \geq \gamma_F(1 - \kappa) \norm{u^{k+2}-S_u(\alpha^{k+1})}^2 - \frac{L_F}{4\kappa}\norm{u^{k+2}-u^{k+1}}^2.
    \]
    Inserting the $u$ update of \cref{alg:FEFB} or \ref{alg:FIFB}, i.e.,
    $
        -\tau^{-1}(u^{k+2}-u^{k+1}) = \grad_{u}F(u^{k+1}; \alpha^{k+1})
    $
    and using the three-point identity \eqref{3-point-identity} we get
    \begin{multline*}
        \frac{1}{2\tau}
        \left(
            \norm{u^{k+2}-S_u(\alpha^{k+1})}^2  +
            \norm{u^{k+2}-u^{k+1}}^2 -
            \norm{u^{k+1}-S_u(\alpha^{k+1})}^2
        \right)
        \\
        \le
        - \gamma_F(1 - \kappa) \norm{u^{k+2}-S_u(\alpha^{k+1})}^2 + \frac{L_F}{4\kappa}\norm{u^{k+2}-u^{k+1}}^2.
    \end{multline*}
    Equivalently
    \[
        \left(1+2\tau\gamma_F(1 - \kappa)\right)\norm{u^{k+2}-S_u(\alpha^{k+1})}^2 + \Bigl(1-\frac{\tau L_F}{2\kappa}
        \Bigr)
        \norm{u^{k+2}-u^{k+1}}^2
        \leq \norm{u^{k+1}-S_u(\alpha^{k+1})}^2.
    \]
    Because \cref{ass:FEFB:main}\,\ref{ass:FEFB:main:inner-step-length} or \ref{ass:FIFB:main}\,\ref{ass:FIFB:main:inner-and-adjoint-step-length}  guarantees $1-\tau L_F/(2\kappa) > 0,$ this implies
    \[
        \norm{u^{k+2}-S_u(\alpha^{k+1})} \leq C_F^{-1} \norm{u^{k+1}-S_u(\alpha^{k+1})}.
    \]
    Therefore the triangle inequality, \eqref{grad:u_u_alpha_closeness_formula} for $n=k$ and the Lipschitz continuity of $S_u$ in $B(\opt{\alpha},2r)\ni \alpha^{k}, \alpha^{k+1}$ yield
    \[
        \begin{aligned}[t]
            \norm{u^{k+2}-S_u(\alpha^{k+1})}
            &
            \le
            C_F^{-1} \norm{u^{k+1}-S_u(\alpha^{k+1})} \leq C_F^{-1} \bigl(
                \norm{u^{k+1}-S_u(\alpha^{k})} + L_{S_u}\norm{\alpha^{k+1}-\alpha^{k}}
            \bigr)
            \\
            &
            \leq C_F^{-1} \bigl(
                C_u \norm{\alpha^{k} - \opt{\alpha}} + L_{S_u}\norm{\alpha^{k+1}-\alpha^{k}}
            \bigr).
        \end{aligned}
    \]
    Inserting \eqref{ineq:S_u_exactness_essential} here, we establish the claim.
\end{proof}

The next lemma is a crucial monotonicity-type estimate for the outer problem.
It depends on an $\alpha$-relative exactness condition on the inner and adjoint variables.

\begin{lemma}\label{lemma:outer-problem-monotonicity}
    Let $n \in \N$.
    Suppose \cref{ass:FEFB:main}\,\cref{ass:FEFB:main:outer-objective,ass:FEFB:main:testing-params}, or \ref{ass:FIFB:main}\,\cref{ass:FIFB:main:outer-objective,ass:FIFB:main:testing-params} hold with \cref{Iterate_condition} and
    \begin{equation}\label{grad:up_up_alpha_closeness_formula_0}
        \norm{u^{n+1}-S_u(\alpha^n)} \leq C_u \norm{\alpha^n - \opt{\alpha}} \,  \text{ and } \,  \onorm{p^{n+1}-\grad_{\alpha}S_u(\alpha^n)} \leq C \norm{\alpha^n - \opt{\alpha}}.
    \end{equation}
    Then, for any $d > 0$,
    \begin{multline}
        \label{alpha_row_general}
        \iprod{p^{n+1}\grad_u J(u^{n+1}) + \subdiff R(\alpha^{n+1})}{\alpha^{n+1}-\opt \alpha} \geq - \frac{L_\alpha}{2} \norm{\alpha^{n+1}- \alpha^n}^2
        \\
        + \left(\frac{\gamma_\alpha}{2}-\frac{L_{\grad J}N_p C_u + C N_{\grad J}}{2d}\right)\norm{\alpha^{n+1}-\opt \alpha}^2
        + \left(\frac{\gamma_\alpha}{2}-\frac{(L_{\grad J}N_p C_u + C N_{\grad J})d}{2}\right)\norm{\alpha^n-\opt \alpha}^2.
    \end{multline}
\end{lemma}

\begin{proof}
    The $\alpha$-update of both \cref{alg:FEFB,alg:FIFB} in implicit form reads
    \[
        0 = \sigma(q^{n+1} + p^{n+1}\grad_u J(u^{n+1}))  + \alpha^{n+1} - \alpha^n
        \quad\text{for some}\quad
        q^{n+1} \in \subdiff R(\alpha^{n+1}).
    \]
    Similarly, $0 \in H(\opt u, \opt p, \opt \alpha)$ implies
    $
        \opt p\,\grad_{u}J(\opt u) +  \opt q =0
    $
    for some
    $
        \opt q \in \subdiff R(\opt \alpha).
    $
    Writing $E_0$ for the left hand side of \eqref{alpha_row_general},
    these expressions and the monotonicity of $\partial R$ yield
    \begin{equation}\label{grad_alpha_iprod}
        \begin{aligned}[t]
            E_0
            &
            =
            \iprod{p^{n+1}\grad_u J(u^{n+1})- \opt p\,\grad_{u}J(\opt u) + q^{n+1} - \opt q}{\alpha^{n+1}-\opt \alpha}
            \\
            &
            \ge
            \iprod{p^{n+1}\grad_u J(u^{n+1})-\grad_{\alpha}S_u(\alpha^n)\grad_u J(S_u(\alpha^n))}{\alpha^{n+1}-\opt \alpha}
            \\
            \MoveEqLeft[-1]
            + \iprod{\grad_{\alpha}S_u(\alpha^n)\grad_u J(S_u(\alpha^n))-\opt p\,\grad_{u}J(\opt u)}{\alpha^{n+1}-\opt \alpha}
            =: E_1 + E_2.
        \end{aligned}
    \end{equation}

    We estimate $E_1$ and $E_2$ separately.
    The one-dimensional mean value theorem gives
    \[
            E_2
            =
            \iprod{\grad_\alpha (J\circ S_u)(\alpha^n)-\grad_\alpha (J\circ S_u)(\opt \alpha)}{\alpha^{n+1}-\opt \alpha}
            =
            \iprod{Q(\alpha^n-\opt \alpha)}{\alpha^{n+1}-\opt \alpha}
    \]
    for some $\zeta \in [\opt \alpha, \alpha^n]$ and $Q \defeq \grad^2_\alpha (J\circ S_u)(\zeta)$.
    Since $\norm{\alpha^n - \opt{\alpha}} \le r$ by \cref{Iterate_condition}, also $\norm{\zeta - \opt \alpha} \le r$.
    Therefore, the 3-point identity \eqref{3-point-identity} and \cref{ass:FEFB:main}\,\ref{ass:FEFB:main:outer-objective} or \ref{ass:FIFB:main}\,\ref{ass:FIFB:main:outer-objective}  yield
    \begin{equation}
        \label{3-point_Hessian_2}
        \begin{aligned}[t]
        E_2
        &
        =
        \frac{1}{2}\norm{\alpha^{n+1}-\opt \alpha}^2_{Q}
        + \frac{1}{2}\norm{\alpha^n-\opt \alpha}^2_{Q}
        - \frac{1}{2}\norm{\alpha^{n+1}- \alpha^n}^2_{Q}
        \\
        &
        \ge
        \frac{\gamma_\alpha}{2}(\norm{\alpha^{n+1}-\opt \alpha}^2 + \norm{\alpha^n-\opt \alpha}^2) - \frac{L_\alpha}{2}\norm{\alpha^{n+1}- \alpha^n}^2.
        \end{aligned}
    \end{equation}

    To estimate $E_1$ we rearrange
    \[
        \begin{aligned}[t]
            E_1
            &
            =
            \iprod{p^{n+1}\grad_u J(u^{n+1})-\grad_{\alpha}S_u(\alpha^n)\grad_u J(S_u(\alpha^n))}{\alpha^{n+1}-\opt \alpha}
            \\
            &
            =
            \iprod{p^{n+1}(\grad_u J(u^{n+1})-\grad_u J(S_u(\alpha^n))) + (p^{n+1}-\grad_{\alpha}S_u(\alpha^n))\grad_u J(S_u(\alpha^n))}{\alpha^{n+1}-\opt{\alpha}}.
        \end{aligned}
    \]
    We have $\norm{\grad_u J(S_u(\alpha^n))} \leq N_{\grad J}$ by the definition of the latter in \cref{ass:FEFB:main}\,\cref{ass:FIFB:main:testing-params} or \ref{ass:FIFB:main}\,\cref{ass:FIFB:main:testing-params} with $\alpha^n\in B(\opt{\alpha}, r)$ from \cref{Iterate_condition}. The same assumptions establish that $\grad_u J$ is Lipschitz. Hence, using the operator norm inequality \cref{thm:separable:properties}\,\ref{item:separable:operator-norm},
    \[
        \begin{aligned}[t]
            E_1
            &
            \geq
            - \onorm{p^{n+1}}\norm{\grad_u J(u^{n+1})-\grad_u J(S_u(\alpha^n))}\norm{\alpha^{n+1}-\opt \alpha}
            \\
            \MoveEqLeft[-1]
            - \onorm{p^{n+1}-\grad_{\alpha}S_u(\alpha^n)}\norm{\grad_u J(S_u(\alpha^n))} \norm{\alpha^{n+1}-\opt \alpha}
            \\
            &
            \geq
            -\left(L_{\grad J}\onorm{p^{n+1}}\norm{u^{n+1}-S_u(\alpha^n)} + C\norm{\grad_u J(S_u(\alpha^n))}\norm{\alpha^n-\opt \alpha}\right) \norm{\alpha^{n+1}-\opt \alpha}.
        \end{aligned}
    \]
    Applying \eqref{grad:up_up_alpha_closeness_formula_0} and Young's inequality now yields for any $d>0$ the estimate
    \begin{equation}
        \label{grad_J_estimate_u}
        \begin{aligned}[t]
            E_1
            &
            \geq
            -\left(L_{\grad J}N_p C_u + C N_{\grad J}\right) \norm{\alpha^n-\opt \alpha}\norm{\alpha^{n+1}-\opt \alpha}
            \\
            &
            \geq
            -\left(L_{\grad J}N_p C_u + C N_{\grad J}\right) \left(\frac{d}{2}\norm{\alpha^n-\opt \alpha}^2 + \frac{1}{2d}\norm{\alpha^{n+1}-\opt \alpha}^2\right).
        \end{aligned}
    \end{equation}
    By inserting \eqref{3-point_Hessian_2} and \eqref{grad_J_estimate_u} into \eqref{grad_alpha_iprod} we obtain the claim \eqref{alpha_row_general}.
\end{proof}

\subsection{Convergence: forward--exact--forward-backward}
\label{subsec:convergence:FEFB}

We now prove the convergence of \cref{alg:FEFB}. We start with a lemma that shows an $\alpha$-relative exactness estimate on the adjoint iterate when one holds for the inner iterate. This is needed to use \cref{lemma:outer-problem-monotonicity}.
The main result of this subsection is in the final \cref{thr:grad-exact-grad_convergence}.
It proves under \cref{ass:FEFB:main} the linear convergence of $\{(u^n, \alpha^n)\}_{n \in \N}$ generated by \cref{alg:FEFB} to $(\opt u, \opt \alpha)$ solving the first-order optimality condition \eqref{eq:main_optimality_condition} for some $\opt p$.

\begin{lemma}\label{lemma:in-eq:p_diff}
    Let $n \in \N$.
    Suppose \cref{ass:FEFB:main} and the inner exactness estimate \eqref{grad:u_u_alpha_closeness_formula} hold as well as $\alpha^{n}\in B(\opt{\alpha}, r_0)$ and $u^{n}\in B(\opt{u}, \sqrt{\sigma^{-1}\inv \phi_u\tau}r_0)$.
    Then \eqref{grad:adjoint_closeness_formula} and \eqref{grad:up_up_alpha_closeness_formula_0} hold for $C=L_{S_p}C_u.$
\end{lemma}

\begin{proof}
    Since \eqref{grad:adjoint_closeness_formula} with  \eqref{grad:u_u_alpha_closeness_formula} equals \eqref{grad:up_up_alpha_closeness_formula_0}, it suffices to prove  \eqref{grad:adjoint_closeness_formula}.
    We assumed $\alpha^n\in B(\opt{\alpha}, r_0)$  and $u^{n+1},S_u(\alpha^n)\in B(\opt{u}, r_u)$ by \cref{lemma:Iterate_condition_corollary}.
    Therefore the Lipschitz continuity of $S_p$ in $B(\opt{u}, r_u)\times B(\opt{\alpha}, r)$ from \cref{lemma:S_p-Lipschitz} with \cref{ass:FEFB:main}\,\cref{ass:FEFB:main:neighbourhoods,ass:FEFB:main:solution-map-and-F} and \eqref{grad:u_u_alpha_closeness_formula} give
    \[
        \onorm{p^{n+1} - \grad_{\alpha}S_u(\alpha^n)}
        =
        \onorm{S_p(u^{n+1},\alpha^n)-S_p(S_u(\alpha^{n}),\alpha^{n})}
        \leq
        L_{S_p}\norm{u^{n+1}-S_u(\alpha^{n})}
        \leq
        L_{S_p} C_u \norm{\alpha^{n} - \opt{\alpha}}.
        \qedhere
    \]
\end{proof}

We are able to collect the previous lemmas into a descent estimate from which we immediately observe local linear convergence.
We recall the definitions of the preconditioning and testing operators $M$ and $Z$ in \eqref{eq:mkplus1} and \eqref{eq:testing-operator}.

\begin{lemma}
    \label{Fejer_lemma_implication}
    Let $n \in \N$ and suppose \cref{ass:FEFB:main,Iterate_condition}, and the inner exactness estimate \eqref{grad:u_u_alpha_closeness_formula} hold.
    Then
    \begin{equation}
        \label{Fejer_monot_u}
        \norm{x^{n+1}-\opt x}_{ZM}^2
        +2\varepsilon_u\norm{u^{n+1}-\opt u}^2 + 2\epsilon_{\alpha}\norm{\alpha^{n+1} - \opt \alpha}^2
        \le
        \norm{x^n-\opt x}_{ZM}^2
    \end{equation}
    for $\phi_u>0$ as in \cref{ass:FEFB:main}\,\cref{ass:FEFB:main:testing-params},
    \[
        \epsilon_u \defeq \frac{\phi_u \gamma_F(1-\kappa)}{2} > 0,
        \quad\text{and}\quad
        \epsilon_{\alpha} \defeq \frac{\gamma_\alpha - (L_{\grad J}N_p + L_{S_p} N_{\grad J})C_u}{2} > 0.
    \]
\end{lemma}

\begin{proof}
    We start by proving the monotonicity estimate
    \begin{equation}
        \label{ineq:FEFB:monotonicity-prove}
        \iprod{ZH_{n+1}(x^{n+1})}{x^{n+1}-\opt{x}} \ge \Vfun_{n+1}(\opt{x}) -  \frac{1}{2}\|x^{n+1}-x^{n}\|^2_{ZM}
    \end{equation}
    for $\Vfun_{n+1}(\opt{u}, \opt{p}, \opt{\alpha}) \defeq \varepsilon_u\norm{u^{n+1}-\opt u}^2 + \epsilon_{\alpha}\norm{\alpha^{n+1} - \opt \alpha}^2$.
    We observe that $\epsilon_u,\epsilon_{\alpha}>0$ by \cref{ass:FEFB:main}.
    The monotonicity estimate \cref{ineq:FEFB:monotonicity-prove} expands as
    \begin{gather}
        \label{ineq:grad-ineq-to-prove_u}
        h_{n+1}
        \geq
        \Vfun_{n+1}(\opt{u}, \opt{p}, \opt{\alpha})
        -\frac{\phi_u}{2\tau}\norm{u^{n+1}-u^{n}}^2 -\frac{1}{2\sigma}\norm{\alpha^{n+1}-\alpha^{n}}^2
    \shortintertext{for (all elements of the set)}
        \nonumber
        h_{n+1} := \left \langle
        \begin{pmatrix}
            \phi_u\grad_{u}F(u^{n};\alpha^{n})  \\
            p^{n+1} \grad_{u}^2 F(u^{n+1};\alpha^{n})  + \grad_{\alpha u}F(u^{n+1};\alpha^{n}) \\
            p^{n+1}\grad_u J(u^{n+1}) + \subdiff R(\alpha^{n+1})
        \end{pmatrix}
        ,
        \begin{pmatrix}
            u^{n+1}-\opt{u} \\
            p^{n+1}-\opt{p}\\
            \alpha^{n+1}-\opt{\alpha}
        \end{pmatrix}
        \right \rangle .
    \end{gather}
    We estimate each of the three lines of $h_{n+1}$ separately.
    For the first line, we use \eqref{u_row_Young:general} from \cref{lemma:inner-problem-monotonicity}.
    For the middle line we observe that $p^{n+1}\grad_{u}^2 F(u^{n+1};\alpha^{n}) + \grad_{\alpha u}F( u^{n+1};\alpha^{n})=0$ by the $p$-update of \cref{alg:FEFB}.

    For the last line, we use \eqref{alpha_row_general} from \cref{lemma:outer-problem-monotonicity} with $d=2$.
    We can do this because \eqref{grad:up_up_alpha_closeness_formula_0} holds by \cref{grad:u_u_alpha_closeness_formula,lemma:in-eq:p_diff}. This gives
    \[
        \iprod{p^{n+1}\grad_u J(u^{n+1}) + \subdiff R(\alpha^{n+1})}{\alpha^{n+1}-\opt \alpha} \geq - \frac{L_\alpha}{2} \norm{\alpha^{n+1}- \alpha^{n}}^2
        + \epsilon_{\alpha}\norm{\alpha^{n+1}-\opt \alpha}^2
        + \epsilon_{\alpha}\norm{\alpha^{n}-\opt \alpha}^2.
    \]
    Summing with \cref{u_row_Young:general} we thus obtain
    \[
        \begin{aligned}[t]
        h_{n+1}
        &
        \geq
        \frac{\phi_u\gamma_F(1 - \kappa)}{2}\norm{u^{n+1}-\opt u}^2 - \frac{\phi_u L_F}{4\kappa}\norm{u^{n+1}-u^{n}}^2  - \frac{L_\alpha}{2} \norm{\alpha^{n+1}- \alpha^{n}}^2
        \\
        \MoveEqLeft[-1]
        + \epsilon_{\alpha} \norm{\alpha^{n+1}-\opt \alpha}^2
        + \left(\epsilon_{\alpha} - \frac{\phi_u L_{\grad F,\opt{u}}^2}{2 \gamma_F(1-\kappa)}\right)\norm{\alpha^{n}-\opt \alpha}^2.
        \end{aligned}
    \]
    The factor of the first term is $\epsilon_u$ and the factor of last term is zero.
    Since  $\sigma <1/L_\alpha$ by  \cref{lemma:sigma_bound_corollary}
    and $L_F/(2\kappa) \le 1/\tau$ by \cref{ass:FEFB:main}\,\ref{ass:FEFB:main:inner-step-length}, we obtain \eqref{ineq:grad-ineq-to-prove_u}, i.e., \eqref{ineq:FEFB:monotonicity-prove}.

    We now come to the fundamental argument of the testing approach of \cite{tuomov-proxtest}, combining operator-relative monotonicity estimates with the three-point identity.
    Indeed, \eqref{ineq:FEFB:monotonicity-prove} combined with the implicit algorithm \eqref{eq:FEFB:implicit} gives
    \[
        \iprod{ZM(x^{n+1} - x^n)}{x^{n+1} - \opt{x}}
        +
        \Vfun_{n+1}(\opt{x})
        \le
        \frac{1}{2}\norm{x^{n +1}-x^{n}}^2_{ZM}.
    \]
    Inserting the three-point identity \eqref{3-point-identity} and expanding $\Vfun_{n+1}$ yields \eqref{Fejer_monot_u}.
\end{proof}

Before stating our main convergence result for the FEFB, we simplify the assumptions of the previous lemma to just \cref{ass:FEFB:main}.

\begin{lemma}
    \label{Fejer_lemma}
    Suppose \cref{ass:FEFB:main} holds.
    Then \eqref{Fejer_monot_u} holds for any $n\in\N$.
\end{lemma}

\begin{proof}
    Then claim readily follows if we prove by induction for all $n \in \N$ that
    \begin{equation}
        \label{eq:FEFB:fejer:induction}
        \tag{*}
        \cref{Iterate_condition}, \eqref{grad:u_u_alpha_closeness_formula}, \text{ and }  \eqref{Fejer_monot_u}
        \text{ hold}.
    \end{equation}

    We first prove \eqref{eq:FEFB:fejer:induction} for $n=0$.
    \Cref{ass:FEFB:main}\,\ref{ass:FEFB:main:init} directly establishes \eqref{grad:u_u_alpha_closeness_formula}.
    The definition of $r_0$ in \cref{ass:FEFB:main} also establishes that $\alpha^{n}\in B(\opt{\alpha}, r_0)$ and $u^{n}\in B(\opt{u}, \sqrt{\sigma^{-1}\inv \phi_u\tau}r_0)$.
    We have just proved the conditions of \cref{lemma:in-eq:p_diff}, which establishes \eqref{grad:adjoint_closeness_formula} for $n=0$.
    Now \cref{lemma:p-np} establishes $\onorm{p^1} \le N_p$.
    Therefore \cref{Iterate_condition} holds for $n=0$.
    Finally \cref{Fejer_lemma_implication} proves \eqref{Fejer_monot_u} for $n=0$.
    This concludes the proof of the induction base.

    We then make the induction assumption that \eqref{eq:FEFB:fejer:induction} holds for $n\in\{0,\ldots,k\}$ and prove it for $n=k+1$.
    Indeed, the induction assumption and \cref{lemma:grad:u_u_alpha_closeness_formula} give \eqref{grad:u_u_alpha_closeness_formula} for $n=k+1$.
    Next \eqref{Fejer_monot_u} for $n=k$ implies $\alpha^{k+1}\in B( \opt{\alpha},r_0)$ and $u^{k+1} \in B(\opt{u}, \sqrt{\sigma^{-1}\inv \phi_u\tau} r_0)$, where $r_0$ and $r_u$ are as in \cref{ass:FEFB:main}.
    Therefore \cref{lemma:Iterate_condition_corollary} gives $u^{k+2}\in B(\opt{u}, r_u)$
    while  \cref{lemma:in-eq:p_diff} establishes \eqref{grad:adjoint_closeness_formula} for $n=k+1$.
    For all $n\in\{0,\ldots,k\}$, the inequality \eqref{Fejer_monot_u} implies $\norm{x^{n+1}-\opt x}_{ZM} \le \norm{x^n-\opt x}_{ZM}$.
    Therefore \cref{lemma:assumptions} proves \cref{Iterate_condition} and finally \cref{Fejer_lemma_implication} proves \eqref{Fejer_monot_u} and consequently \eqref{eq:FEFB:fejer:induction} for $n=k+1$.
\end{proof}

\begin{theorem}\label{thr:grad-exact-grad_convergence}
    Suppose \cref{ass:FEFB:main} holds. Then $\phi_u\tau^{-1}\norm{u^{n}-\opt u}^2 + \sigma^{-1}\norm{\alpha^{n} - \opt \alpha}^2 \to 0$ linearly.
\end{theorem}
\begin{proof}
    \Cref{Fejer_lemma}, expansion of \eqref{Fejer_monot_u}, and basic manipulation show that
    \begin{align*}
        \phi_u\tau^{-1}\norm{u^{n}-\opt u}^2 + \sigma^{-1}\norm{\alpha^{n} - \opt \alpha}^2
        &
        \ge
        (\phi_u\tau^{-1}+2\varepsilon_u)\norm{u^{n+1}-\opt u}^2 + (\sigma^{-1}+2\epsilon_{\alpha})\norm{\alpha^{n+1} - \opt \alpha}^2
        \\
        &
        = (1+2\varepsilon_u\inv \phi_u\tau)\phi_u\tau^{-1}\norm{u^{n+1}-\opt u}^2 + (1+2\epsilon_{\alpha}\sigma)\sigma^{-1}\norm{\alpha^{n+1} - \opt \alpha}^2
        \\
        &
        \ge
        \mu \bigl(\phi_u\tau^{-1}\norm{u^{n+1}-\opt u}^2 + \sigma^{-1}\norm{\alpha^{n+1} - \opt \alpha}^2\bigr)
    \end{align*}
    for $\mu := \min\{1+2\varepsilon_u\inv \phi_u \tau, 1+2\epsilon_{\alpha}\sigma\}$.
    Since $\mu >1$, linear convergence follows.
\end{proof}

\subsection{Convergence: forward--inexact--forward-backward}
\label{subsec:convergence:FIFB}

We now prove the convergence of \cref{alg:FIFB}.
The overall structure and idea of the proofs follows \cref{subsec:convergence:FEFB} and uses several lemmas from \cref{subsec:convergence:general}.
We first prove monotonicity estimate lemma for the adjoint step and then
that a small enough step length in the outer problem gurantees that the inner and adjoint iterates stay in a small local neighbourhood if they are already in one.
The main result of this subsection is in the final \cref{thr:grad-grad-grad_convergence}.
It proves under \cref{ass:FIFB:main} the linear convergence of $\{(u^n, p^n, \alpha^n)\}_{n \in \N}$ generated by \cref{alg:FIFB} to $(\opt u, \opt p, \opt \alpha)$ solving the first-order optimality condition \eqref{eq:main_optimality_condition}.

\begin{lemma}\label{lemma:FIFB-adjoint:monotonicity}
	Let $u\in U, \alpha\in\AlphaSpace$ and $p_1, p_2, \tilde p \in P.$ Moreover, $\gamma_F\cdot\Id\le\grad_u^2 F(u; \alpha)\le L_F\cdot\Id$ and
	\begin{equation}\label{eq:lemma-adjoint-equation}
		\tilde p \grad_u^2 F(u; \alpha) + \grad_{\alpha u}F(u; \alpha) = 0.
	\end{equation}
	holds. Then
	\begin{align*}
		\oiprod{p_1 \grad_u^2 F(u; \alpha) + \grad_{\alpha u}(u; \alpha)}{p_2 - \tilde p} 
		\ge
		\frac{\gamma_F}{2}\onorm{p_2 - \tilde p}^2
		+  \frac{\gamma_F}{2}\onorm{p_1 - \tilde p}^2
		- \frac{L_F}{2}\onorm{p_2- p_1}^2.
	\end{align*}
	
\end{lemma}
\begin{proof}
	 Using \cref{eq:lemma-adjoint-equation}, the three-point identity \eqref{3-point-identity} and $\gamma_F\cdot\Id\le\grad_u^2 F(u; \alpha)\le L_F\cdot\Id$ gives
     \[
	\begin{aligned}
		&
		\oiprod{ p_1 \grad_{u}^2 F(u; \alpha) +\grad_{\alpha u} F(u; \alpha)}{p_2-\tilde p}
		\\
        &
		=
		\oiprod{(p_1-\tilde p)\grad_u^2 F(u; \alpha)}{p_2-\tilde p}
		\\
        &
		=
		\sum_{i\in I} \iprod{\grad_{u}^2 F(u; \alpha)(p_1-\tilde p)^*\phi_i}{(p_2-\tilde p)^*\phi_i}
		\\
        &
		=
		\sum_{i\in I}
		\left(
		\frac{1}{2} \norm{(p_1-\tilde p)^*\phi_i}^2_{\grad_{u}^2 F(u; \alpha)}
		-
		\frac{1}{2}\norm{(p_2-p_1)^*\phi_i} ^2_{\grad_{u}^2 F(u; \alpha)}
		+
		\frac{1}{2}\norm{(p_2-\tilde p)^*\phi_i} ^2_{\grad_{u}^2 F(u; \alpha)}
		\right)
		\\
		&
		\geq
		\sum_{i\in I}\left(
		\frac{\gamma_F}{2} \norm{(p^{k+1}-\tilde p)^*\phi_i}^2
		-
		\frac{L_F}{2}\norm{(p^{k+2}-p^{k+1})^*\phi_i} ^2
		+
		\frac{\gamma_F}{2}\norm{(p^{k+2}-\tilde p)^*\phi_i} ^2
		\right)
		\\
		&
		=
		\frac{\gamma_F}{2}\onorm{p_2 - \tilde p}^2
		+  \frac{\gamma_F}{2}\onorm{p_1 - \tilde p}^2
		- \frac{L_F}{2}\onorm{p_2- p_1}^2.
	\end{aligned}
    \]
\end{proof}

\begin{lemma}\label{lemma:grad:up_up_alpha_closeness_formula}
    Let $k \in \N$. Suppose \cref{ass:FIFB:main} holds, and \cref{Iterate_condition} and
    \begin{equation}
        \label{grad:up_up_alpha_closeness_formula}
        \norm{u^{n+1}-S_u(\alpha^{n})} \leq C_u \norm{\alpha^{n} - \opt{\alpha}} \, \text{ and } \, \onorm{p^{n+1}-\grad_{\alpha}S_u(\alpha^{n})} \leq C_p \norm{\alpha^{n} - \opt{\alpha}}
    \end{equation}
    hold for $n=k$.
    Then \eqref{grad:up_up_alpha_closeness_formula} holds for $n = k+1.$
\end{lemma}
\begin{proof}
    Observe that \eqref{grad:up_up_alpha_closeness_formula} for $n=k$ implies \eqref{grad:u_u_alpha_closeness_formula}  as well as \eqref{grad:adjoint_closeness_formula} for $n=k$ and $C=C_p$.
    \Cref{lemma:grad:u_u_alpha_closeness_formula} therefore proves the first part of \eqref{grad:up_up_alpha_closeness_formula} for $n=k+1$, i.e.,
    \begin{equation}
        \label{grad:up_up_alpha_closeness_formula:p1}
        \norm{u^{k+2}-S_u(\alpha^{k+1})} \leq C_u \norm{\alpha^{k+1} - \opt{\alpha}}.
    \end{equation}
        
    We still need to show the second part $\onorm{p^{k+2}-\grad_{\alpha}S_u(\alpha^{k+1})} \leq C_p \norm{\alpha^{k+1} - \opt{\alpha}}$.
    We follow the fundamental argument of the testing approach (see the end of the proof of \cref{Fejer_lemma_implication}) and use \cref{ass:FIFB:main}\,\cref{ass:FIFB:main:neighbourhoods,ass:FIFB:main:solution-map-and-F}.
    For the latter we need $\alpha^k, \alpha^{k+1}\in B(\opt{\alpha}, 2r)$ and $u^{k+2}, S_u(\alpha^{k})\in B(\opt{u}, r_u).$  We have $\alpha^{k}\in B(\opt{\alpha},r_0)$ by \cref{Iterate_condition} and $\alpha^{k+1}\in B(\opt{\alpha},2r_0)$ by \cref{lemma:grad:u_u_alpha_closeness_formula}. Thus we may use the Lipschitz continuity of $S_u$ with the triangle inequality and \eqref{grad:up_up_alpha_closeness_formula:p1} to get $S_u(\alpha^k)\in B(S_u(\opt\alpha), L_{S_u}r_0)\subset B(\opt{u}, r_u)$ and
    \[
        \begin{aligned}
        \norm{u^{k+2}-\opt{u}}
        &
        \leq \norm{u^{k+2}-S_u(\alpha^{k+1})} + \norm{S_u(\alpha^{k+1}) - S_u(\opt{\alpha})}
        \\
        &
        \leq (C_u+ L_{S_u})\norm{\alpha^{k+1} - \opt{\alpha} } \leq (C_u+ L_{S_u})2r_0,
        \end{aligned}
    \]
	which yields $u^{k+2}\in B(\opt u, r_u).$
    The definition of $S_p$ in \eqref{def:S_p} implies
    \[
    	S_p(u^{k+2}, \alpha^{k+1})\grad_{u}^2 F(u^{k+2}; \alpha^{k+1}) + \grad_{\alpha u} F(u^{k+2}; \alpha^{k+1})=0.
    \]
    Since also
    $\gamma_F\cdot\Id\le\grad_u^2 F\le L_F\cdot\Id$
    in $B(\opt{u}, r_u) \times B(\opt{\alpha}, 2r)$ from \cref{ass:FIFB:main}\,\ref{ass:FIFB:main:neighbourhoods}, we get
    \begin{equation}
  	\label{ineq:FIFB-p-step-monotonicity}
    \begin{aligned}[t]
        &
        \oiprod{ p^{k+1} \grad_{u}^2 F(u^{k+2}; \alpha^{k+1}) +\grad_{\alpha u} F(u^{k+2}; \alpha^{k+1})}{p^{k+2}-S_p(u^{k+2},\alpha^{k+1})}
        \\
        &
        \geq
        \frac{\gamma_F}{2}\onorm{p^{k+2}-S_p(u^{k+2},\alpha^{k+1})}^2
        + \frac{\gamma_F}{2}\onorm{p^{k+1}-S_p(u^{k+2},\alpha^{k+1})}^2
        - \frac{L_F}{2}\onorm{p^{k+2} - p^{k+1}}^2 =: A
    \end{aligned}
    \end{equation}
    from \cref{lemma:FIFB-adjoint:monotonicity}.
    By the $p$ update of the FIFB in the implicit form \eqref{eq:FIFB:implicit0}, we have
    \[
        p^{k+1} \grad_{u}^2 F(u^{k+2}; \alpha^{k+1}) +\grad_{\alpha u} F(u^{k+2}; \alpha^{k+1})
        = -\inv\theta(p^{k+2}-p^{k+1}).
    \]
    Combining with \eqref{ineq:FIFB-p-step-monotonicity} gives
    \[
        \begin{aligned}[t]
            \MoveEqLeft[1]
            -\oiprod{p^{k+2}-p^{k+1}}{p^{k+2}-S_p(u^{k+2},\alpha^{k+1})}
            \ge
            \theta A.
        \end{aligned}
    \]
    An application of the three-point identity \eqref{3-point-identity} with $\theta L_F \le 1$ from \cref{ass:FIFB:main}\,\cref{ass:FIFB:main:inner-and-adjoint-step-length} now yields for $C_{F,S} = \sqrt{(1+\theta \gamma_F)/(1-\theta \gamma_F)}$ the estimate
    \[
        \onorm{p^{k+2}-S_p(u^{k+2},\alpha^{k+1})} \leq C_{F,S}^{-1} \onorm{p^{k+1}-S_p(u^{k+2},\alpha^{k+1})}.
    \]
    This estimate and the triangle inequality give
    \begin{equation}
  	\label{grad:adjoint_closeness_bound}
    \begin{aligned}
        \onorm{p^{k+2}-\grad_{\alpha}S_u(\alpha^{k+1})}
        &
        =
        \onorm{p^{k+2}- S_p(S_u(\alpha^{k+1}),\alpha^{k+1})}
        \\
        \MoveEqLeft[8]
        \leq \onorm{p^{k+2}-S_p(u^{k+2},\alpha^{k+1})} + \onorm{S_p(u^{k+2},\alpha^{k+1}) - S_p(S_u(\alpha^{k+1}),\alpha^{k+1})}
        \\
        \MoveEqLeft[8]
        \leq C_{F,S}^{-1} \onorm{p^{k+1}-S_p(u^{k+2},\alpha^{k+1})} + \onorm{S_p(u^{k+2},\alpha^{k+1}) - S_p(S_u(\alpha^{k+1}),\alpha^{k+1})}
        =: E_1 + E_2.
	\end{aligned}
    \end{equation}
    
    The solution map $S_u$ is Lipschitz in $B(\opt{\alpha}, 2r)$ and $S_p$ is Lipschitz in $B(\opt{u}, r_u)\times B(\opt{\alpha}, 2r)$ due to \cref{ass:FIFB:main}\,\cref{ass:FIFB:main:neighbourhoods,ass:FIFB:main:solution-map-and-F}, and \cref{lemma:S_p-Lipschitz}.
    Combined with the triangle inequality, \eqref{grad:up_up_alpha_closeness_formula} for $n=k$ and \eqref{grad:up_up_alpha_closeness_formula:p1}, we obtain
    \begin{equation}
  	\label{grad:adjoint_closeness_bound_E1}
    \begin{aligned}[t]
    		C_{F,S} E_1
    		&
    		\leq
    		\onorm{p^{k+1}-S_p(S_u(\alpha^{k}),\alpha^{k})} + \onorm{S_p(u^{k+2},\alpha^{k+1})-S_p(S_u(\alpha^{k}),\alpha^{k})}
    		\\
    		&
    		\leq
            \onorm{p^{k+1}-\grad_{\alpha} S_u(\alpha^{k}))} + L_{S_p} \left(\norm{u^{k+2} - S_u(\alpha^{k})} + \norm{\alpha^{k+1} -  \alpha^{k}}\right)
    		\\
            &
    		\leq
    		C_p \norm{\alpha^{k} - \opt{\alpha}}  + L_{S_p} \left(\norm{u^{k+2} - S_u(\alpha^{k+1})} + \norm{S_u(\alpha^{k+1})-S_u(\alpha^{k})} + \norm{\alpha^{k+1} -  \alpha^{k}}\right)
    		\\
    	    &
    		\leq
    		E_3 + L_{S_p} C_u \norm{\alpha^{k+1} - \opt{\alpha}}
    \end{aligned}
    \end{equation}
    for
    \[
        E_3
        \defeq C_p \norm{\alpha^{k} - \opt{\alpha}} + L_{S_p}(1+L_{S_u})\norm{\alpha^{k+1} -  \alpha^{k}}.
    \]
  	Using again the Lipschitz continuity of $S_p$ and \eqref{grad:up_up_alpha_closeness_formula:p1}, we get
    \begin{equation}
    	\label{grad:adjoint_closeness_bound_E2}
			E_2
            \leq
            L_{S_p}\norm{u^{k+2} - S_u(\alpha^{k+1})}
            \leq
            L_{S_p} C_u \norm{\alpha^{k+1} -  \opt{\alpha}} .
    \end{equation}
	Inserting \cref{grad:adjoint_closeness_bound_E1,grad:adjoint_closeness_bound_E2} into \cref{grad:adjoint_closeness_bound} yields
	\[
        \onorm{p^{k+2}-\grad_{\alpha}S_u(\alpha^{k+1})}
        \leq
        C_{F,S}^{-1} E_3
        +
        (C_{F,S}^{-1} + 1)
        L_{S_p} C_u \norm{\alpha^{k+1} -  \opt{\alpha} }.
	\]
    Therefore the claim follows if we show that
    \begin{equation}\label{ineq:sigma_adjoint_upper_bound}
            C_{F,S}^{-1} E_3
            \leq (C_p - (C_{F,S}^{-1}+1) L_{S_p} C_u) \norm{\alpha^{k+1} - \opt{\alpha}}.
    \end{equation}
    
    \Cref{lemma:alpha-option} proves \eqref{ineq:alpha_in_exact_grad} with $C=C_p$.
    Together with \cref{ass:FIFB:main}\,\cref{ass:FIFB:main:outer-step-length} it yields
    \[
        \norm{\alpha^{k+1} - \alpha^k} \leq \sigma C_{\alpha} \norm{\alpha^{k} - \opt{\alpha}}  \leq \frac{(C_{F,S}-1)C_p- (1+C_{F,S})L_{S_p}C_u}{(1+L_{S_u})L_{S_p}+C_{F,S}C_p- (1+C_{F,S})L_{S_p}C_u} \norm{\alpha^{k} - \opt{\alpha}}.
    \]
    Multiplying by $(1+L_{S_u})L_{S_p}+C_{F,S}C_p- (1+C_{F,S})L_{S_p}C_u,$ rearranging terms, and continuing with the triangle inequality gives \eqref{ineq:sigma_adjoint_upper_bound}. Indeed,
    \begin{align*}
            E_3
            &
            \leq
            C_{F,S}(C_p - (C_{F,S}^{-1}+1) L_{S_p} C_u)(\norm{\alpha^{k} - \opt{\alpha}} - \norm{\alpha^{k+1} - \alpha^{k}})
            \\
            &
            \leq
            C_{F,S}(C_p - (C_{F,S}^{-1}+1) L_{S_p} C_u)\norm{\alpha^{k+1} - \opt{\alpha}}.
            \qedhere
    \end{align*}
\end{proof}

We now show that the adjoint iterates stay local if the outer iterates do.

Again, by combining the previous lemmas, we prove an estimate from which local convergence is immediate. For this, we recall the definitions of the preconditioning and testing operators $M$ and $Z$ in \eqref{eq:mkplus1_FIFB} and \eqref{eq:testing-operator}.

\begin{lemma}\label{Fejer_lemma_up_implication}
    Suppose \cref{ass:FIFB:main,Iterate_condition}, and the inner and adjoint exactness estimate \eqref{grad:up_up_alpha_closeness_formula} hold for $n\in\N.$ Then
    \begin{equation}
        \label{Fejer_monot_up}
        \norm{x^{n+1}-\opt x}_{ZM}^2
        + 2\varepsilon_u\norm{u^{n+1}-\opt u}^2 + 2\varepsilon_p\norm{p^{n}-\opt p}^2 + 2\varepsilon_{\alpha}\norm{\alpha^{n+1} - \opt{\alpha}}^2
        \leq
        \norm{x^n-\opt x}_{ZM}^2
    \end{equation}
    for
    \[
    	\epsilon_u \defeq \frac{\phi_u\gamma_F(1-\kappa)}{2}-C_{\alpha,1}
    	>
    	0,
    	\quad
    	\epsilon_p \defeq \frac{\phi_p \gamma_F}{2}
    	>
    	0,
    	\quad\text{and}\quad
    	\epsilon_{\alpha} \defeq
    	\frac{\gamma_\alpha -  C_{\alpha,1} -  \sqrt{C_{\alpha,1}^2+4C_{\alpha,2}^2}}{2}
    	>
    	0
    \]
    where $\phi_u, \phi_p>0$ are as in \cref{ass:FIFB:main},
    $C_{\alpha,1} \defeq \phi_p\tfrac{L_FL_{S_p}}{\gamma_F}$,
    and $C_{\alpha,2} \defeq \bigl(L_{\grad J}N_p + L_{S_p} N_{\grad J}\bigr) \tfrac{C_u}{2}$.
\end{lemma}

\begin{proof}
    We start by proving the monotonicity estimate
    \begin{equation}
        \label{ineq:grad-ineq-to-prove:0}
        \iprod{ZH_{n+1}(x^{n+1})}{x^{n+1}-\opt{x}}
        \ge
        \Vfun_{n+1}(\opt{x}) -  \frac{1}{2}\norm{x^{n+1}-x^{n}}^2_{ZM}
    \end{equation}
    for $\Vfun_{n+1}(\opt{u}, \opt{p}, \opt{\alpha}) = \varepsilon_u\norm{u^{n+1}-\opt u}^2 + \varepsilon_p\norm{p^{n}-\opt p}^2 + \varepsilon_{\alpha}\norm{\alpha^{n+1} - \opt \alpha}^2$.
    We observe that $\epsilon_u,\epsilon_p,\epsilon_{\alpha}>0$ by \cref{ass:FIFB:main}.
    The monotonicity estimate \cref{ineq:grad-ineq-to-prove:0} expands as
    \begin{gather}
        \label{ineq:grad-ineq-to-prove}
        h_{n+1} \geq
        \Vfun_{n+1}(\opt{u}, \opt{p}, \opt{\alpha})
        -\frac{\phi_u}{2\tau}\norm{u^{n+1}-u^{n}}^2
        -\frac{\phi_p}{2\theta}\onorm{p^{n+1}-p^{n}}^2
        -\frac{1}{2\sigma}\norm{\alpha^{n+1}-\alpha^{n}}^2
    \shortintertext{for (all elements of the set)}
        \nonumber
        h_{n+1}:= \left \langle
        \begin{pmatrix}
            \phi_u\grad_{u}F(u^{n};\alpha^{n})  \\
            \phi_p \left(p^{n} \grad_{u}^2 F(u^{n+1};\alpha^{n})  + \grad_{\alpha u}F(u^{n+1};\alpha^{n})\right) \\
            p^{n+1}\grad_u J(u^{n+1}) + \subdiff R(\alpha^{n+1})
        \end{pmatrix}
        ,
        \begin{pmatrix}
            u^{n+1}-\opt{u} \\
            p^{n+1}-\opt{p}\\
            \alpha^{n+1}-\opt{\alpha}
        \end{pmatrix}
        \right \rangle .
    \end{gather}
    We estimate the three lines of $h_{n+1}$ separately.
    We immediately take care of the first line by using \eqref{u_row_Young:general} from \cref{lemma:inner-problem-monotonicity}.

    For the second line, using
    the optimality condition \eqref{eq:p-row-oc}
    we have
    \begin{equation}
    	\label{eq:FIFB-adjoint-main-1}
    	\begin{aligned}[t]
    		&
    		\phi_p\oiprod{ p^{n} \grad_{u}^2 F(u^{n+1}; \alpha^{n}) +\grad_{\alpha u} F(u^{n+1}; \alpha^{n})}{p^{n+1}-\widehat{p}}
    		\\
    		&
    		=
    		\phi_p\oiprod{ (p^{n}-S_p(u^{n+1},\alpha^{n})) \grad_{u}^2 F(u^{n+1}; \alpha^{n})}{p^{n+1}-\widehat{p}}
    		\\
    		&
    		=
    		\phi_p\oiprod{(p^{n}-\widehat{p} ) \grad_{u}^2 F(u^{n+1}; \alpha^{n})}{p^{n+1}-\widehat{p}}
    		+
    		\phi_p\oiprod{(\widehat{p}-S_p(u^{n+1},\alpha^{n})) \grad_{u}^2 F(u^{n+1}; \alpha^{n})}{p^{n+1}-\widehat{p}}
    		\\
    		&
    		=: \phi_p(E_1 + E_2).
    	\end{aligned}
    \end{equation}
  	We have  $u^{n+1}$, $S_u(\alpha^n)\in B(\opt{u}, r_u)$, and $\alpha^n\in B(\opt{\alpha}, r)$ by \cref{lemma:Iterate_condition_corollary,Iterate_condition}.
  	Thus
  	   $\gamma_F\cdot\Id\le\grad_u^2 F\le L_F\cdot\Id$ 
  	in $B(\opt{u}, r_u) \times B(\opt{\alpha}, 2r)$ and $\norm{\grad_{u}^2 F(u^{n+1};\alpha^{n})} \le L_F$ by \cref{ass:FIFB:main}\,\cref{ass:FIFB:main:neighbourhoods}. We get
  	\begin{equation}
  		\label{ineq:adjoint-FIFB-main-2}
  		E_1
  		\geq
  		\frac{\gamma_F}{2}\onorm{p^{n+1}-\widehat{p}}^2
  		+ \frac{\gamma_F}{2}\onorm{p^{n}-\widehat{p}}^2
  		- \frac{L_F}{2}\onorm{p^{n+1} - p^{n}}^2
  	\end{equation}
  	from \cref{lemma:FIFB-adjoint:monotonicity}. By \cref{thm:separable:properties}\,\cref{item:separable:norm-iprod}
    $\oiprod{\freevar}{\freevar}$ is an inner product and $\onorm{\freevar}$ a norm on $\linear(U; \AlphaSpace),$ and we can use thus Cauchy–Schwarz inequality for them. Therefore, using also \cref{thm:separable:properties}\,\cref{item:separable:operator-norm-combo}, \cref{lemma:S_p-Lipschitz} and Young's inequality, we can estimate
    \begin{equation}
    	\label{ineq:adjoint-FIFB-main-3}
    	\begin{aligned}[t]
    		E_2
    		&
    		\geq
    		- \big|	\oiprod{(\widehat{p}-S_p(u^{n+1},\alpha^{n})) \grad_{u}^2 F(u^{n+1}; \alpha^{n})}{p^{n+1}-\widehat{p}} \big|
    		\\
    		&
    		\geq
    		- \onorm{(\widehat{p}-S_p(u^{n+1},\alpha^{n})) \grad_{u}^2 F(u^{n+1}; \alpha^{n})} \onorm{p^{n+1}-\widehat{p}}
    		\\
    		&
    		\geq
    		- L_F\onorm{S_p(\widehat{u}, \widehat{\alpha})-S_p(u^{n+1},\alpha^{n}))}  \onorm{p^{n+1}-\widehat{p}}
    		\\
    		&
    		\geq
    		 - L_FL_{S_p}\left(\norm{u^{n+1}-\widehat{u}}+ \norm{\alpha^{n} - \widehat{\alpha}} \right) \onorm{p^{n+1}-\widehat{p}}
    		 \\
    		 &
    		 \geq
    		 - \frac{L_FL_{S_p}}{\gamma_F}\left(\norm{u^{n+1}-\widehat{u}}^2 + \norm{\alpha^{n} - \widehat{\alpha}}^2 \right) - \frac{\gamma_F}{2}\onorm{p^{n+1}-\widehat{p}}^2.
    	\end{aligned}
    \end{equation}

	Inserting \cref{ineq:adjoint-FIFB-main-2,ineq:adjoint-FIFB-main-3}
	into
	\cref{eq:FIFB-adjoint-main-1},
	we obtain
	\begin{equation}
		\label{ineq:FIFB-adjoint-main}
		\begin{aligned}[t]
			&
			\phi_p\oiprod{ p^{n} \grad_{u}^2 F(u^{n+1}; \alpha^{n}) +\grad_{\alpha u} F(u^{n+1}; \alpha^{n})}{p^{n+1}-\widehat{p}}
			\\
			&
			\geq
			\frac{\phi_p\gamma_F}{2}\onorm{p^{n}-\widehat{p}}^2
			- \frac{\phi_pL_F}{2}\onorm{p^{n+1} - p^{n}}^2
			- \frac{\phi_p L_FL_{S_p}}{\gamma_F}\left(\norm{u^{n+1}-\widehat{u}}^2 + \norm{\alpha^{n} - \widehat{\alpha}}^2 \right)
		\end{aligned}
	\end{equation}
	
    The factor of the last term equals $C_{\alpha,1}$ from \cref{ass:FIFB:main}\,\cref{ass:FIFB:main:outer-objective}.
    
    Since \cref{Iterate_condition} and \eqref{grad:up_up_alpha_closeness_formula} hold, \cref{lemma:outer-problem-monotonicity} gives \eqref{alpha_row_general} with $C = C_p$ and any $d>0$ for the third line of $h_{n+1}$.
    Summing \cref{u_row_Young:general,ineq:FIFB-adjoint-main,alpha_row_general} we finally deduce
    \[
        \begin{aligned}[t]
        h_{n+1}
        &
        \ge
        \left(\frac{\phi_u\gamma_F(1 - \kappa)}{2}-C_{\alpha,1}\right)\norm{u^{n+1}-\opt u}^2
        - \frac{\phi_u L_F}{4\kappa}\norm{u^{n+1}-u^{n}}^2 + \frac{\phi_p\gamma_F}{2}\onorm{p^{n}-\widehat{p}}^2
        - \frac{\phi_pL_F}{2}\onorm{p^{n+1}-p^{n}}^2
        \\
        \MoveEqLeft[-1]
        - \frac{L_\alpha}{2} \norm{\alpha^{n+1}- \alpha^{n}}^2
        + \biggl(\frac{\gamma_\alpha}{2}- \frac{C_{\alpha,2}}{d} \biggr) \norm{\alpha^{n+1}-\opt \alpha}^2
        +\biggl(\frac{\gamma_\alpha}{2} - \frac{\phi_u L_{\grad F,\opt{u}}^2}{2\gamma_F(1-\kappa)} -  C_{\alpha,1} - C_{\alpha,2}d\biggr)\norm{\alpha^{n}-\opt \alpha}^2.
        \end{aligned}
    \]
    We have
    \[
        \frac{C_{\alpha,2}}{d} = C_{\alpha,1} + C_{\alpha,2}d = \frac{C_{\alpha,1}}{2} + \frac{\sqrt{C_{\alpha,1}^2+4C_{\alpha,2}^2}}{2}
        \quad\text{for}\quad
        d = \frac{-C_{\alpha,1} + \sqrt{C_{\alpha,1}^2+4C_{\alpha,2}^2}}{2C_{\alpha,2}}.
    \]
    Then also $\frac{\gamma_\alpha}{2}- \frac{C_{\alpha,2}}{d}=\epsilon_{\alpha}$.
    It follows
    \[
        h_{n+1}
        \ge
        \epsilon_u\norm{u^{n+1}-\opt u}^2
        - \frac{\phi_u L_F}{4\kappa}\norm{u^{n+1}-u^{n}}^2
        + \epsilon_p\norm{p^{n}-\opt p}^2
        - \frac{\phi_p L_F}{2}\onorm{p^{n+1}-p^{n}}^2
        - \frac{L_\alpha}{2} \norm{\alpha^{n+1}- \alpha^{n}}^2
        + \epsilon_{\alpha} \norm{\alpha^{n+1}-\opt \alpha}^2.
    \]
    Since $\sigma <1/L_\alpha$ by  \cref{lemma:sigma_bound_corollary},
    $L_F/(2\kappa) \le 1/\tau$ and $\theta < 1/L_F$ by \cref{ass:FIFB:main}\,\ref{ass:FIFB:main:inner-and-adjoint-step-length}, we obtain \eqref{ineq:grad-ineq-to-prove}, i.e., \eqref{ineq:grad-ineq-to-prove:0}.
    We finish by applying the fundamental arguments of the testing approach to \eqref{ineq:grad-ineq-to-prove:0} and the general implicit update \eqref{eq:FEFB:implicit} as in \cref{Fejer_lemma_implication}.
\end{proof}

We simplify the assumptions of the previous lemma to just \cref{ass:FIFB:main}.

\begin{lemma}\label{Fejer_lemma_up}
    Suppose \cref{ass:FIFB:main} holds. Then
    \eqref{Fejer_monot_up} holds for any $n\in\N$.
\end{lemma}

\begin{proof}
    The claim readily follows if
    we prove by	induction for all $n\in\N$ that
    \begin{equation}
        \label{eq:FIFB:fejer:induction}
        \tag{**}
        \cref{Iterate_condition}, \eqref{grad:up_up_alpha_closeness_formula}, \text{ and }  \eqref{Fejer_monot_up}
        \text{ hold}.
    \end{equation}
    We first prove \eqref{eq:FIFB:fejer:induction} for $n=0.$
    \Cref{ass:FIFB:main}\,\cref{ass:FIFB:main:init} directly establishes \eqref{grad:up_up_alpha_closeness_formula}.
    The definition of $r_0$ in \cref{ass:FIFB:main} also establishes that $\alpha^n\in B(\opt \alpha, r_0)$ and $u^n\in B(\opt u, \sqrt{\sigma^{-1}\inv \phi_u\tau}r_0).$
    We have just proved the conditions of \cref{lemma:p-np}, which gives $\onorm{p^1}\le N_p$. Thus we we proved \cref{Iterate_condition} for $n=0.$
    Now \cref{Fejer_lemma_up_implication} proves \eqref{Fejer_monot_up} for $n=0.$ This concludes the proof of induction base.
    
    We then make the induction assumption that \eqref{eq:FIFB:fejer:induction} holds for $n\in\{0,\ldots,k\}$
    and prove it for $n=k+1.$
    The induction assumption and \cref{lemma:grad:up_up_alpha_closeness_formula} give \eqref{grad:up_up_alpha_closeness_formula} for $n=k+1.$  The inequality  \eqref{Fejer_monot_up} for	$n\in\{0,\ldots,k\}$  also ensures $\norm{x^{n+1}-\opt x}_{ZM} \le \norm{x^{n}-\opt x}_{ZM}$ for $n\in\{0,\ldots,k\}.$  Therefore \cref{lemma:assumptions} proves \cref{Iterate_condition} for $n=k+1$.
    Now \cref{Fejer_lemma_up_implication} shows \eqref{Fejer_monot_up} and concludes the proof of \cref{eq:FIFB:fejer:induction} for $n=k+1$.
\end{proof}

We finally come to the main convergence result for the FIFB.

\begin{theorem}\label{thr:grad-grad-grad_convergence}
	Suppose \cref{ass:FIFB:main} holds. Then $\phi_u\tau^{-1}\norm{u^{n}-\opt u}^2 + \phi_p\theta^{-1}\norm{p^{n}-\opt p}^2 + \sigma^{-1}\norm{\alpha^{n} - \opt \alpha}^2 \to 0$ linearly.
\end{theorem}
\begin{proof}
	We define $\mu_1:= \min\{(1+2\varepsilon_u\inv \phi_u \tau), (1+2\varepsilon_{\alpha}\sigma)\}$ and $\mu_2:= 1 - 2\varepsilon_p\inv \phi_p\theta.$
	\Cref{Fejer_lemma_up}, expansion of \eqref{Fejer_monot_up}, and basic manipulation show that
	\begin{equation}
		\label{ineq:FIFB-linear}
		\begin{aligned}[t]
			&
			\mu_1 \bigl(\phi_u\tau^{-1}\norm{u^{n+1}-\opt u}^2 + \sigma^{-1}\norm{\alpha^{n+1} - \opt \alpha}^2\bigr) + \phi_p\inv\theta\onorm{p^{n+1}-\opt p}^2
			\\
			&
			\leq
			(1+2\varepsilon_u\inv \phi_u\tau)\phi_u\tau^{-1}\norm{u^{n+1}-\opt u}^2 + (1+2\varepsilon_{\alpha}\sigma)\sigma^{-1}\norm{\alpha^{n+1} - \opt \alpha}^2 + \phi_p\inv\theta\onorm{p^{n+1}-\opt p}^2
			\\
			&
			=
			(\phi_u\tau^{-1}+2\varepsilon_u)\norm{u^{n+1}-\opt u}^2 + (\sigma^{-1}+2\varepsilon_{\alpha})\norm{\alpha^{n+1} - \opt \alpha}^2 + \phi_p\inv\theta\onorm{p^{n+1}-\opt p}^2
			\\
			&
			\leq
			\phi_u\tau^{-1}\norm{u^{n}-\opt u}^2 + \mu_2\phi_p\theta^{-1}\onorm{p^{n}-\opt p}^2 + \sigma^{-1}\norm{\alpha^{n} - \opt \alpha}^2.
		\end{aligned}
	\end{equation}
	There are two separate cases (i) $\mu_1 \mu_2 \leq 1$ and (ii) $\mu_1 \mu_2 > 1.$ In case (i), we have
	\begin{equation*}
		\begin{aligned}
			&
			\phi_u\tau^{-1}\norm{u^{n}-\opt u}^2 + \mu_2\phi_p\theta^{-1}\onorm{p^{n}-\opt p}^2 + \sigma^{-1}\norm{\alpha^{n} - \opt \alpha}^2
			\\
			&
			=
			\inv\mu_1
			\bigl(
			\mu_1
			\bigl(
			\phi_u\tau^{-1}\norm{u^{n}-\opt u}^2 + \sigma^{-1}\norm{\alpha^{n} - \opt \alpha}^2
			\bigr)
			+ \mu_1\mu_2\phi_p\theta^{-1}\onorm{p^{n}-\opt p}^2
			\bigr)
			\\
			&
			\leq
			\inv\mu_1
			\bigl(
			\mu_1
			\bigl(
			\phi_u\tau^{-1}\norm{u^{n}-\opt u}^2 + \sigma^{-1}\norm{\alpha^{n} - \opt \alpha}^2
			\bigr)
			+ \phi_p\theta^{-1}\onorm{p^{n}-\opt p}^2
			\bigr),
		\end{aligned}
	\end{equation*}
	which with \cref{ineq:FIFB-linear} implies linear convergence since $\inv\mu_1\in (0,1).$ In case (ii), we obtain
	\begin{equation*}
		\begin{aligned}
			&
			\phi_u\tau^{-1}\norm{u^{n}-\opt u}^2 + \mu_2\phi_p\theta^{-1}\onorm{p^{n}-\opt p}^2 + \sigma^{-1}\norm{\alpha^{n} - \opt \alpha}^2
			\\
			&
			\leq
			\mu_2
			\bigl(
			\mu_1
			\bigl(
			\phi_u\tau^{-1}\norm{u^{n}-\opt u}^2 +
			 \sigma^{-1}\norm{\alpha^{n} - \opt \alpha}^2
			\bigr)
			+
			\phi_p\theta^{-1}\onorm{p^{n}-\opt p}^2
			\bigr),
		\end{aligned}
	\end{equation*}
	which with \cref{ineq:FIFB-linear} implies linear convergence since $\mu_2\in (0,1).$
\end{proof}

\section{Numerical experiments}
\label{sec:numerical}

We evaluate the performance of our proposed algorithms on parameter learning for (anisotropic, smoothed) total variation image denoising and deconvolution. For a “ground truth” image $b \in \R^{N^2}$ of dimensions $N \times N$, we take
\[
	J(u) = \frac{1}{2}\norm{u-b}_2^2
\]
as the outer fitness function. For $b$ we use a cropped portion of image 02 or 08 from the free Kodak dataset \cite{franzenkodak} converted to gray values in $[0, 1]$. %
The purpose of these numerical experiments is a simple performance comparison between our proposed methods and a few representative approaches from the literature. We therefore only consider a single ground-truth image $b$ and a corresponding corrupted data $z$ in the various inner problems, which we next describe. For proper generalizable parameter learning, multiple such training pairs $(b_i, z_i)$ should be used. This can in principle be done by summing over all the data in both the inner and outer problem; resulting in a higher-dimensional bilevel problem; see, e.g., \cite{calatroni2015bilevel}.
In practise, a large sample count would require stochastic techniques.

\subsection{Denoising}

For denoising we take in \eqref{eq:bilevel-problem} as the inner objective
\[%
    F(u; \alpha) = \frac{1}{2}\norm{u - z}_2^2 + \alpha \norm{Du}_{1,\gamma}
    \quad (u \in \R^{N^2}, \alpha \in [0, \infty)),
\]%
and as the outer regulariser $R \equiv 0$.
The simulated measurement $z$ is obtained from $b$ by adding Gaussian noise of standard deviation $0.1$.
The matrix $D$ is a backward difference operator with Dirichlet boundary conditions. Instead of the one-norm, $\norm{\cdot}_1$, to ensure the twice differentiability of the objective and hence a simple adjoint equation, we use a $C^2$ Huber- or Moreau--Yosida -type approximation with
\[
    \norm{y}_{1, \gamma} := \sum_{i=1}^{2N^2} \rho_{\gamma}(y_i),
    \quad\text{where}\quad
	\rho_{\gamma}(x) \defeq
	\begin{cases}
		-\frac{|x|^3}{3\gamma^2} + \frac{|x|^2}{\gamma} & \text{if } \, |x|\leq \gamma, \\
		|x| - \frac{\gamma}{3} & \text{if } \, |x| > \gamma.
	\end{cases}
\]
We used $\gamma=10^{-4}$ in our experiments.
\begin{figure}[t]
    \centering
    \begin{subfigure}[t]{0.24\textwidth}
        \centering
        \includegraphics[width=\textwidth]{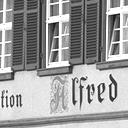}
        \caption{Original}
        \label{original}
    \end{subfigure}%
    \hfill%
    \begin{subfigure}[t]{0.24\textwidth}
        \centering
        \includegraphics[width=\textwidth]{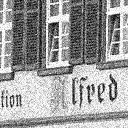}
        \caption{Noisy; error 14.1\%}
        \label{noisy}
    \end{subfigure}%
    \hfill%
    \begin{subfigure}[t]{0.24\textwidth}
        \centering
        \includegraphics[width=\textwidth]{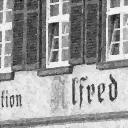}
        \caption{Implicit; error 9.3\%}
        \label{exact}
    \end{subfigure}%
    \hfill%
    \begin{subfigure}[t]{0.24\textwidth}
        \centering
        \includegraphics[width=\textwidth]{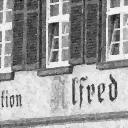}
        \caption{FIFB; error 9.2\%}
        \label{GGG}
    \end{subfigure}
    \caption{Denoising data and results for the implicit and FIFB methods ($N=128$).}
    \label{fig:denoising}
\end{figure}

\subsection{Deconvolution}

For deconvolution, we take as the inner objective
\[%
    F(u; \alpha) = \frac{1}{2}\norm{K(\alpha) * u-z}_2^2 + C\alpha_1 \norm{Du}_{1,\gamma},
    \quad (u \in \R^{N^2}, \alpha \in [0, \infty)^4 ),
\]
and as the outer regulariser $R(\alpha)=\beta(\sum_{i=2}^4\alpha_i-1)^2+\delta_{[0, \infty)}(\alpha_1)$ for a regularisation parameter $\beta=10^4$.
We introduce the constant $C=\tfrac{1}{10}$ to help convergence by ensuring the same order of magnitude for all components of $\alpha$.
The first element of $\alpha$ is the total variation regularization parameter while the rest parametrizes the convolution kernel $K(\alpha)$ as illustrated in \cref{fig:numerical:deblurring-param}. The sum of the elements of the kernel equals $\alpha_2 + \alpha_3 + \alpha_4.$
Operator $r_{\theta}$ rotates image $\theta$ degrees, clockwise for $\theta>0$ and counterclocwise for $\theta<0.$
We form $z$ by computing $r_{-1}(K(\alpha)*r_1(b))$ for $\alpha = [0.15, 0.1, 0.75] $ and adding Gaussian noise of standard deviation $1\cdot 10^{-2}$.

\begin{figure}[t]
    \begin{subfigure}[t]{0.24\textwidth}%
        \resizebox{\linewidth}{!}{%
    	\begin{tikzpicture}[inner sep=0pt,outer sep=0pt]%
    		\fill[gray!40!white] (0.6,0) rectangle (2.4,3.0);
    		\fill[gray!40!white] (0,0.6) rectangle (3.0,2.4);
    		\fill[blue!40!white] (0.6,1.2) rectangle (2.4,1.8);
    		\fill[blue!40!white] (1.2,0.6) rectangle (1.8,2.4);
    		\fill[red!40!white] (1.2,1.2) rectangle (1.8,1.8);
    		\node at (1.5, 1.5) {$\alpha_2$};
    		\node at (1.5, 2.1) {$\alpha_3$};
    		\node at (2.1, 2.1) {$\alpha_4$};
    		\node at (0.3, 0.3) {$0$};
    		\node at (0.3, 2.7) {$0$};
    		\node at (2.7, 0.3) {$0$};
    		\node at (2.7, 2.7) {$0$};
    		\draw[step=0.6cm,black,very thin] (0,0) grid (3.0,3.0);
    	\end{tikzpicture}%
        }%
    	\caption{Kernel structure}
		\label{fig:numerical:deblurring-param}
	\end{subfigure}%
    \hfill%
	\begin{subfigure}[t]{0.24\textwidth}%
		\centering
		\includegraphics[width=\textwidth]{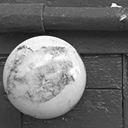}
		\caption{Original}
		\label{original-blur}
	\end{subfigure}%
    \hfill%
    \begin{subfigure}[t]{0.24\textwidth}
    	\centering
    	\includegraphics[width=\textwidth]{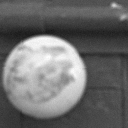}
    	\caption{Blurry; error $9.8\%$}
    	\label{blurred-noisy}
    \end{subfigure}%
    \hfill%
    \begin{subfigure}[t]{0.24\textwidth}
    	\centering
    	\includegraphics[width=\textwidth]{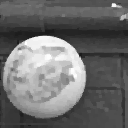}
    	\caption{FIFB; error $7.2\%$}
    	\label{FIFB-blur}
    \end{subfigure}%
    \caption{Deconvolution kernel parametrisation, data, and result for FIFB ($N=128$).}
    \label{fig:deblurring}
\end{figure}
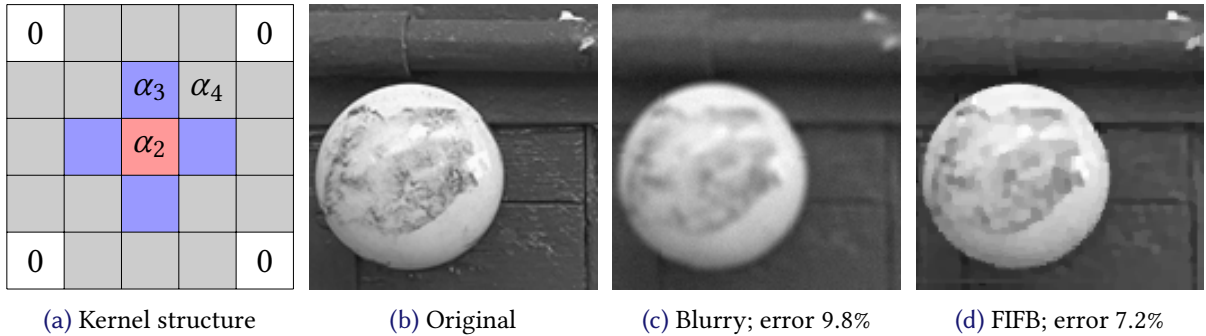

For denoising, and deconvolution assuming $\ker D \isect \ker K(\alpha) = \{0\}$, it is not difficult to verify the structural parts of \cref{ass:FEFB:main,ass:FIFB:main}, required for the convergence results of \cref{sec:convergence}.
We do not attempt to verify the conditions on the step lengths, choosing them by trial and error.

\subsection{An implicit baseline method}

We evaluate \cref{alg:FEFB,alg:FIFB} against a conventional method that solves both the inner problem and the adjoint equation (near-)exactly, as well as the AID \cite{ji2021bilevel}.
We also experimented with solving the equivalent constrained optimisation problem $\min_{\alpha, u} J(u)$ subject to $\grad_u F(u; \alpha)=0$ with IPOPT \cite{wacher2002interior} and the NL-PDPS \cite{tuomov-nlpdhgm-redo,tuomov-nlpdhgm}. However, we did not observe convergence without the inclusion of additional $H^1$ regularisation in the inner problem, as in, e.g., \cite{delosreyes2014learning}. Since that changes the problem, we have not included “whole problem” approaches in our comparison.

To solve the inner problem in the implicit baseline method, we use gradient descent, starting at $v^0=0$ and updating $v^{m+1} \defeq v^m - \tau_m \grad F(v^m; \alpha^k)$ We then set $u^{k+1}=v^{m+1}$. 
The adjoint and outer iterate updates are as in \cref{alg:FEFB}, however, we discover $\sigma=\sigma_k$ by the line search rule \cite[(12.41)]{clason2020introduction} for nonsmooth problems, starting at $\sigma_k=5\cdot 10^{-5}$ and multiplying by $0.1$ on each line search step.
For deconvolution we use a fixed step length parameter, as it performed better.
The specific parameter choices (step lengths, number of inner and adjoint iterations) for all algorithms and experiments are listed in \cref{tab:setup}.

\begin{table}[t]
    \caption{Algorithm parametrisation (step length parameters, inner and adjoint iteration counts), time multiplier, and outer steps taken to reach threshold computational resources (CPU time) value.
    The threshold is 6000 for denoising, and 15000 for deconvolution. The time multipliers allows conversion from computational resources to real time.
    It differs between algorithms and problem dimensions due to different levels of parallelisability.
    The 'inner steps' and 'adjoint steps' columns indicate the (maximum) number of iterations taken towards a solution of the inner problem or the adjoint equation on every outer iteration.
    The \texttt{bicgstab} method for solving the adjoint for the implicit method and the FEFB may use a smaller number of iterations as determined by the threshold $10^{-5}$.
    “LS” for the step length parameter $\sigma$ means that line search was used.}
    \label{tab:setup}
	\centering
	\begin{NiceTabular}{clrrrrrrrr}
		 &   &        & outer & inner & adjoint & time &     &          &           \\
		 & N & method & steps & steps & steps & mult. & $\tau$ & $\theta$ & $\sigma$  \\
		\toprule
		\Block{8-1}{\rotate denoising}
		 & 128 & implicit & 15 & $2.5\cdot10^4$ & $10^3$ & $0.22$ & $5\cdot 10^{-4}$ & - & LS  \\
		 &  & AID  & $7.7\cdot 10^3$ & $10$ & $50$ & $0.12$ & $5\cdot 10^{-4}$ & - & $2\cdot 10^{-8}$ \\
		 &  & FEFB & $3.9\cdot 10^3$ & $1$ & $10^3$ & $0.11$ &  $8\cdot 10^{-4}$ & - & $5\cdot 10^{-9}$  \\
		 &  & FIFB & $5.1\cdot 10^5$ & $1$ & $1$ & $0.22$ & $5\cdot 10^{-4}$ & $1\cdot 10^{-5}$ & $5\cdot 10^{-10}$ \\
		 & 256 & implicit & 12 & $2.5\cdot10^4$ & $10^3$ & $0.40$ & $5\cdot 10^{-4}$ & - & LS  \\
		 &  & AID  & $3.9\cdot 10^3$ & $10$ & $50$ & $0.21$ & $5\cdot 10^{-4}$ & - & $1\cdot 10^{-8}$ \\
		 & & FEFB & $2.3\cdot 10^3$ & $1$ & $10^3$ & $0.16$ &  $8\cdot 10^{-4}$ & - & $2\cdot 10^{-9}$  \\
		 & & FIFB & $2.3\cdot 10^5$ & $1$ & $1$ & $0.35$ &  $5\cdot 10^{-4}$  & $7.5\cdot 10^{-6}$ & $5\cdot 10^{-10}$  \\
		 \midrule
		 \Block{8-1}{\rotate deconvolution}
		 & 128 & implicit & $29$ & $1\cdot10^5$ & $10^3$ & $0.21$ &  $5\cdot 10^{-4}$ & - & $5\cdot 10^{-5}$   \\
		 &  & AID  & $11\cdot 10^3$ & $50$ & $50$ & $0.14$ & $5\cdot 10^{-4}$ & - & $2.5\cdot 10^{-6}$ \\
		 &  & FEFB & $4.9\cdot 10^2$ & $1$ & $10^3$ & $0.15$ &  $5\cdot 10^{-4}$ & - & $1\cdot 10^{-6}$  \\
		 &  & FIFB & $5.2\cdot 10^5$ & $1$ & $1$ & $0.22$ &  $5\cdot 10^{-4}$ & $1\cdot 10^{-4}$ & $2.5\cdot 10^{-7}$  \\
		 & 32 & implicit & $8.7\cdot 10^2$ & $5\cdot10^4$ & $10^3$ & $0.93$ & $5\cdot 10^{-4}$ & - & $5\cdot 10^{-6}$  \\
		 &  & AID  & $9.2\cdot 10^5$ & $10$ & $50$ & $1.00$ & $5\cdot 10^{-4}$ & - & $1\cdot 10^{-7}$ \\
		 & & FEFB & $11\cdot 10^3$ & $1$ & $10^3$ & $0.19$ &  $5\cdot 10^{-4}$ & - & $1\cdot 10^{-6}$ \\
		 & & FIFB & $4.4\cdot 10^6$ & $1$ & $1$ & $0.65$ &  $5\cdot 10^{-4}$ & $5\cdot 10^{-5}$ & $7\cdot 10^{-8}$  \\
		\bottomrule
	\end{NiceTabular}
\end{table}

\subsection{Numerical setup}

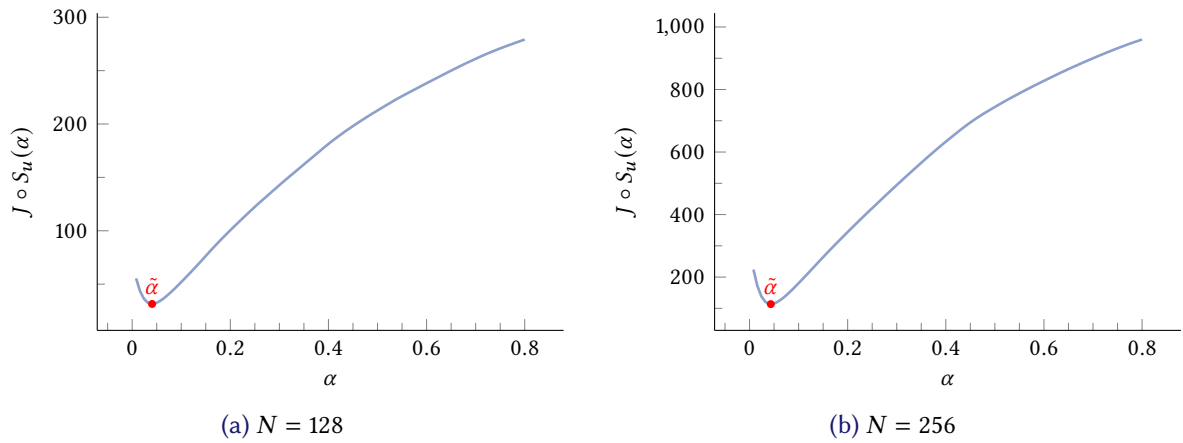
\begin{figure}[t]
	\centering
	\begin{subfigure}[b]{0.49\textwidth}
	\centering
	\begin{tikzpicture}
		\begin{axis}[%
			width=\linewidth,
			height=0.75\linewidth,
			xlabel near ticks,
			ylabel near ticks,
			scaled y ticks=false,
			yminorticks=true,
			minor y tick num=1,
			xminorticks=true,
			minor x tick num=3,
			axis x line*=bottom,
			axis y line*=left,
			xlabel={$\alpha$},
			ylabel={$J\circ S_u(\alpha)$},
			y tick label style={/pgf/number format/fixed},
			outer sep=0pt,
			font=\footnotesize,
			]
			\addplot [color=Set2-C, line width=1pt] table[x=alpha,y=Js]{data/denoising/JsGrid128.txt};
            \fill[color=red] (axis cs:0.04058, 31.4627) circle(1.5pt) node [above] {$\tilde\alpha$};
		\end{axis}
	\end{tikzpicture}
	\caption{$N=128$}
\end{subfigure}%
\hfill%
\begin{subfigure}[b]{0.49\textwidth}
	\centering
	\begin{tikzpicture}
		\begin{axis}[%
			width=\linewidth,
			height=0.75\linewidth,
			xlabel near ticks,
			ylabel near ticks,
			scaled y ticks=false,
			yminorticks=true,
			minor y tick num=1,
			xminorticks=true,
			minor x tick num=3,
			axis x line*=bottom,
			axis y line*=left,		
			xlabel={$\alpha$},
			ylabel={$J\circ S_u(\alpha)$},
			y tick label style={/pgf/number format/fixed},
			outer sep=0pt,
			font=\footnotesize,
			]
			\addplot [color=Set2-C, line width=1pt] table[x=alpha,y=Js]{data/denoising/JsGrid256.txt};
            \fill[color=red] (axis cs:0.04336, 113.5073) circle(1.5pt) node [above] {$\tilde\alpha$};
\end{axis}
\end{tikzpicture}
\caption{$N=256$}
\end{subfigure}%
	\caption{Graph of composed objective the denoising problem (both problem sizes), along with near-optimal $\tilde\alpha$ found by recursive subdivision.}
	\label{fig:numerical:J_grid128}
\end{figure}

Our algorithm implementations are available on Zenodo \cite{suonpera2023codes}.
To solve the adjoint equation in the FEFB and implicit methods, we use Matlab's \verb|bicgstab| implementation of the stabilized biconjugate gradient method \cite{van1992bi} with tolerance $10^{-5}$, and maximum iteration count $10^{3}$.
With the AID we use 50 conjugate gradient iterations.
These choices, as well as the choice of the number of inner iterations for the implicit method and the AID, have been made by trial and error to be as small as possible while obtaining an apparently stable algorithm.

To evaluate scalability, we consider for denoising both $N=128$ and $N=256$.
For deconvolution we consider $N=128$ and $N=32.$ 
We take initial $u^0 = S_u(\alpha^0)$ and $p^0 = S_p(u^0, \alpha^0)$ where for denoising $\alpha^0=0$ and for deconvolution  $\alpha^{0}=[0.4, 0.25, 0.25, 0.5]$ and $\alpha^{0}=[0.04, 0.25, 0.25, 0.5]$ with $N=128$ and $N=32$ respectively.

To compare algorithm performance, we plot relative errors against the \verb!cputime! value of Matlab on an AMD Ryzen 5 5600H CPU. We call this value “computational resources”, as it takes into account the use of several CPU cores by Matlab's internal linear algebra, making it fairer than the actual running time.
For each algorithm and problem, we indicate in \cref{tab:setup} the step length parameters, the number of outer steps to reach the computational resources value of 6000 for denoising 15000 for deconvolution , and an average multiplier to convert computational resources into running times.

For performance comparison, we need estimates $\tilde\alpha$ and $\tilde u$ of optimal $\hat\alpha$ and $\hat u=S_u(\hat\alpha)$.
For denoising we find them by searching for the one-dimensional variable $\alpha$ on a regular grid and recursively subdividing until node spacing goes below $10^{-5}$.
As $\tilde u$, we take an estimate of $S_u(\tilde\alpha)$ obtained with 25000 steps of the implicit base line method.
We visualise so obtained $\tilde\alpha$ and $J \circ S_u$ in \cref{fig:numerical:J_grid128}.
For the higher-dimensional deconvolution problem, such a scan is not feasible. Instead, we obtain the comparison estimates by running the implicit method from a very good initial iterate until computational resources (CPU time) value of 6000 for $N=32$ and 10000 for $N=128$. Specifically, we initialise the kernel parameters $(\alpha_2,\alpha_3,\alpha_4)$ as those used for generating the data $z$, and the regularization parameter $\alpha_1 = 0.045$ for $N=32$ and $\alpha_1=0.02$ for $N=128$, the latter found by trial and error.
This initialisation is different from that used for the actual numerical experiments; see above.
Our experiments indicate that the other methods approach $\tilde\alpha$ so obtained faster than the implicit method itself, providing some justification for the choice.

With these solution estimates we define the inner and outer relative errors

\[
    e_{\alpha,\text{rel}} \defeq \tfrac{\norm{\tilde\alpha - \alpha^k}}{\norm{\tilde\alpha}}
    \quad\text{and}\quad
    e_{u,\text{rel}} \defeq \tfrac{\norm{\tilde u - u^k}}{\norm{\tilde u}}.
\]

\subsection{Results}

\pgfplotsset{every axis legend/.code={\renewcommand\addlegendentry[2][]{}}}

\begin{figure}[t]
	\centering
	\begin{subfigure}[b]{0.45\textwidth}
	\centering
	\begin{tikzpicture}
		\begin{axis}[%
			width=\linewidth,
			height=0.75\linewidth,
			ymode=log,
			xlabel near ticks,
			ylabel near ticks,
			scaled y ticks=false,
			yminorticks=true,
			minor y tick num=1,
			xminorticks=true,
			minor x tick num=3,
			axis x line*=bottom,
			axis y line*=left,
			legend columns=4,
			legend style={at={(0.5, 1.0)}, anchor=south,inner sep=0pt,outer sep=0pt,legend cell align=left,align=left,draw=none,fill=none,font=\footnotesize},
			xlabel={computational resources},
			ylabel={$e_{\alpha,\text{rel}}$},
			y tick label style={/pgf/number format/fixed},
			outer sep=0pt,
			font=\footnotesize,
			]
			
			\addplot [color=Set2-A, line width=1pt] table[x=cputime,y=alphaDiff]{data/denoising/bilevelDenoiseFIFB128.txt};
            \label{plot:FIFB}
			\addlegendentry{FIFB}
			
			\addplot [color=Set2-B, line width=1pt] table[x=cputime,y=alphaDiff]{data/denoising/bilevelDenoiseFEFB128.txt};
            \label{plot:FEFB}
			\addlegendentry{FEFB}
			
			\addplot [color=Set2-D, line width=1pt] table[x=cputime,y=alphaDiff]{data/denoising/bilevelDenoiseAID128.txt};
            \label{plot:AID}
			\addlegendentry{AID}
			
			\addplot [color=Set2-C, line width=1pt] table[x=cputime,y=alphaDiff]{data/denoising/bilevelDenoiseImplicit128.txt};
            \label{plot:implicit}
			\addlegendentry{implicit}
		\end{axis}
	\end{tikzpicture}
	\caption{Outer problem, $N=128$}
	\label{fig:numerical:denoising:outer}
\end{subfigure}%
\hfill%
\begin{subfigure}[b]{0.45\textwidth}
	\centering
	\begin{tikzpicture}
		\begin{axis}[%
			width=\linewidth,
			height=0.75\linewidth,
			ymode=log,
			xlabel near ticks,
			ylabel near ticks,
			scaled y ticks=false,
			yminorticks=true,
			minor y tick num=1,
			xminorticks=true,
			minor x tick num=3,
			axis x line*=bottom,
			axis y line*=left,
			legend columns=4,
			legend style={at={(0.5, 1.0)}, anchor=south,inner sep=0pt,outer sep=0pt,legend cell align=left,align=left,draw=none,fill=none,font=\footnotesize},
			xlabel={computational resources},
			ylabel={$e_{u,\text{rel}}$},
			y tick label style={/pgf/number format/fixed},
			outer sep=0pt,
			font=\footnotesize,
			]
			
			\addplot [color=Set2-A, line width=1pt] table[x=cputime,y=uDiff]{data/denoising/bilevelDenoiseFIFB128.txt};
			\addlegendentry{FIFB}
			
			\addplot [color=Set2-B, line width=1pt] table[x=cputime,y=uDiff]{data/denoising/bilevelDenoiseFEFB128.txt};
			\addlegendentry{FEFB}
			
			\addplot [color=Set2-D, line width=1pt] table[x=cputime,y=uDiff]{data/denoising/bilevelDenoiseAID128.txt};
			\addlegendentry{AID}
			
			\addplot [color=Set2-C, line width=1pt] table[x=cputime,y=uDiff]{data/denoising/bilevelDenoiseImplicit128.txt};
			\addlegendentry{implicit}
		\end{axis}
	\end{tikzpicture}
	\caption{Inner problem, $N=128$}
	\label{fig:numerical:denoising:inner}
\end{subfigure}%

\medskip

\begin{subfigure}[b]{0.45\textwidth}
	\centering
	\begin{tikzpicture}
		\begin{axis}[%
			width=\linewidth,
			height=0.75\linewidth,
			ymode=log,
			xlabel near ticks,
			ylabel near ticks,
			scaled y ticks=false,
			yminorticks=true,
			minor y tick num=1,
			xminorticks=true,
			minor x tick num=3,
			axis x line*=bottom,
			axis y line*=left,
			legend columns=4,
			legend style={at={(0.5, 1.0)}, anchor=south,inner sep=0pt,outer sep=0pt,legend cell align=left,align=left,draw=none,fill=none,font=\footnotesize},
			xlabel={computational resources},
			ylabel={$e_{\alpha,\text{rel}}$},
			y tick label style={/pgf/number format/fixed},
			outer sep=0pt,
			font=\footnotesize,
			]
			
			\addplot [color=Set2-A, line width=1pt] table[x=cputime,y=alphaDiff]{data/denoising/bilevelDenoiseFIFB256.txt};
			\addlegendentry{FIFB}
			
			\addplot [color=Set2-B, line width=1pt] table[x=cputime,y=alphaDiff]{data/denoising/bilevelDenoiseFEFB256.txt};
			\addlegendentry{FEFB}
			
			\addplot [color=Set2-D, line width=1pt] table[x=cputime,y=alphaDiff]{data/denoising/bilevelDenoiseAID256.txt};
			\addlegendentry{AID}
			
			\addplot [color=Set2-C, line width=1pt] table[x=cputime,y=alphaDiff]{data/denoising/bilevelDenoiseImplicit256.txt};
			\addlegendentry{implicit}
		\end{axis}
	\end{tikzpicture}
	\caption{Outer problem, $N=256$}
	\label{fig:numerical:denoising:outer256}
\end{subfigure}%
\hfill%
\begin{subfigure}[b]{0.45\textwidth}
	\centering
	\begin{tikzpicture}
		\begin{axis}[%
			width=\linewidth,
			height=0.75\linewidth,
			ymode=log,
			xlabel near ticks,
			ylabel near ticks,
			scaled y ticks=false,
			yminorticks=true,
			minor y tick num=1,
			xminorticks=true,
			minor x tick num=3,
			axis x line*=bottom,
			axis y line*=left,
			legend columns=4,
			legend style={at={(0.5, 1.0)}, anchor=south,inner sep=0pt,outer sep=0pt,legend cell align=left,align=left,draw=none,fill=none,font=\footnotesize},
			xlabel={computational resources},
			ylabel={$e_{u,\text{rel}}$},
			y tick label style={/pgf/number format/fixed},
			outer sep=0pt,
			font=\footnotesize,
			]
	
		\addplot [color=Set2-A, line width=1pt] table[x=cputime,y=uDiff]{data/denoising/bilevelDenoiseFIFB256.txt};
		\addlegendentry{FIFB}
	
		\addplot [color=Set2-B, line width=1pt] table[x=cputime,y=uDiff]{data/denoising/bilevelDenoiseFEFB256.txt};
		\addlegendentry{FEFB}
		
		\addplot [color=Set2-D, line width=1pt] table[x=cputime,y=uDiff]{data/denoising/bilevelDenoiseAID256.txt};
		\addlegendentry{AID}
			
		\addplot [color=Set2-C, line width=1pt] table[x=cputime,y=uDiff]{data/denoising/bilevelDenoiseImplicit256.txt};
		\addlegendentry{implicit}
\end{axis}
\end{tikzpicture}
\caption{Inner problem, $N=256$}
\label{fig:numerical:denoising:inner256}
\end{subfigure}%
	\caption{Denoising performance.
    The graphs correspond to the
    FIFB (\ref*{plot:FIFB}),
    FEFB (\ref*{plot:FEFB}),
    AID (\ref*{plot:AID}), and
    implicit method (\ref*{plot:implicit}).
	}
	\label{fig:numerical:denoising}
\end{figure}
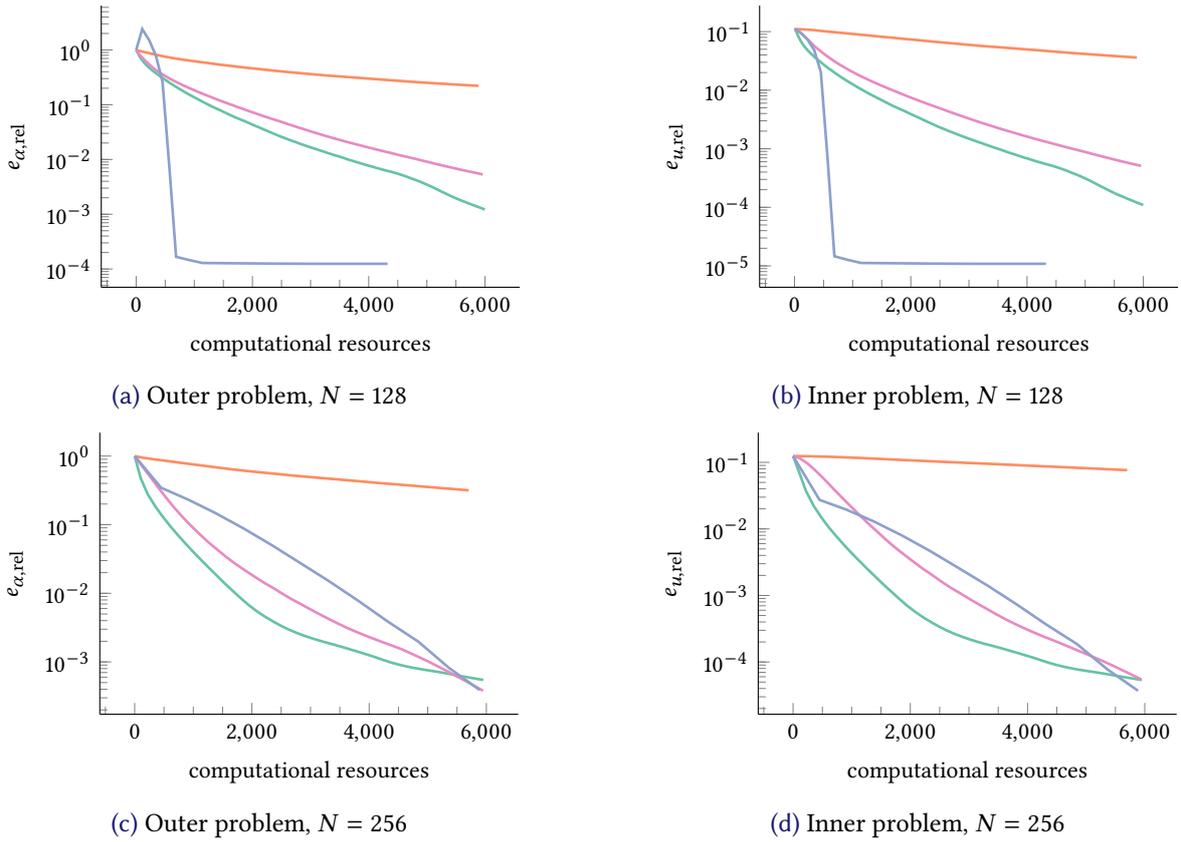

\begin{figure}[t]
	\centering
	\begin{subfigure}[b]{0.45\textwidth}
	\begin{tikzpicture}
		\begin{axis}[%
			width=\linewidth,
			height=0.75\linewidth,
			xlabel near ticks,
			ylabel near ticks,
			scaled y ticks=false,
			yminorticks=true,
			minor y tick num=1,
			xminorticks=true,
			minor x tick num=3,
			axis x line*=bottom,
			axis y line*=left,
 			legend columns=4,
 			legend style={at={(0.5, 1.0)}, anchor=south,inner sep=0pt,outer sep=0pt,legend cell align=left,align=left,draw=none,fill=none,font=\footnotesize},
			xlabel={computational resources},
			ylabel={$e_{\alpha,\text{rel}}$},
			y tick label style={/pgf/number format/fixed},
			outer sep=0pt,
			font=\footnotesize,
			]
						
			\addplot [color=Set2-A, line width=1pt] table[x=cputime,y=alphaDiff]{data/deblurring/bilevelDeblurFIFB32.txt};
            \label{plot:dc:FIFB}
			\addlegendentry{FIFB}
			
			\addplot [color=Set2-B, line width=1pt] table[x=cputime,y=alphaDiff]{data/deblurring/bilevelDeblurFEFB32.txt};
            \label{plot:dc:FEFB}
			\addlegendentry{FEFB}
			
			\addplot [color=Set2-D, line width=1pt] table[x=cputime,y=alphaDiff]{data/deblurring/bilevelDeblurAID32.txt};
            \label{plot:dc:AID}
			\addlegendentry{AID}
			
			\addplot [color=Set2-C, line width=0.7pt] table[x=cputime,y=alphaDiff]{data/deblurring/bilevelDeblurImplicit32.txt};
            \label{plot:dc:implicit}
			\addlegendentry{implicit}
		\end{axis}
	\end{tikzpicture}
	\caption{Outer problem, $N=32$}
	\label{fig:numerical:outer-deblur32}
\end{subfigure}%
\hfill%
\begin{subfigure}[b]{0.45\textwidth}
	\begin{tikzpicture}
		\begin{axis}[%
			width=\linewidth,
			height=0.75\linewidth,
			xlabel near ticks,
			ylabel near ticks,
			scaled y ticks=false,
			yminorticks=true,
			minor y tick num=1,
			xminorticks=true,
			minor x tick num=3,
			axis x line*=bottom,
			axis y line*=left,
 			legend columns=4,
 			legend style={at={(0.5, 1.0)}, anchor=south,inner sep=0pt,outer sep=0pt,legend cell align=left,align=left,draw=none,fill=none,font=\footnotesize},
			xlabel={computational resources},
			ylabel={$e_{u,\text{rel}}$},
			y tick label style={/pgf/number format/.cd,fixed,precision=3},
            ylabel shift=-1ex,
			outer sep=0pt,
			font=\footnotesize,
			]
			
			\addplot [color=Set2-A, line width=1pt] table[x=cputime,y=uDiff]{data/deblurring/bilevelDeblurFIFB32.txt};
			\addlegendentry{FIFB}
			
			\addplot [color=Set2-A, dashed, line width=0.7pt] table[x=cputime,y=SuDiff]{data/deblurring/bilevelDeblurFIFB32Su.txt};
			\label{plot:dc:FIFB-S_u}
			
			\addplot [color=Set2-B, line width=1pt] table[x=cputime,y=uDiff]{data/deblurring/bilevelDeblurFEFB32.txt};
			\addlegendentry{FEFB}
			
			\addplot [color=Set2-D, line width=1pt] table[x=cputime,y=uDiff]{data/deblurring/bilevelDeblurAID32.txt};
			\addlegendentry{AID}
			
			\addplot [color=Set2-C, line width=1pt] table[x=cputime,y=uDiff]{data/deblurring/bilevelDeblurImplicit32.txt};
			\addlegendentry{implicit}
		\end{axis}
	\end{tikzpicture}
	\caption{Inner problem, $N=32$}
	\label{fig:numerical:inner-deblur32}
\end{subfigure}%

\medskip

\begin{subfigure}[b]{0.45\textwidth}
	\begin{tikzpicture}
		\begin{axis}[%
			width=\linewidth,
			height=0.75\linewidth,
			xlabel near ticks,
			ylabel near ticks,
			scaled y ticks=false,
			yminorticks=true,
			minor y tick num=1,
			xminorticks=true,
			minor x tick num=3,
			axis x line*=bottom,
			axis y line*=left,
			legend columns=4,
			legend style={at={(0.5, 1.0)}, anchor=south,inner sep=0pt,outer sep=0pt,legend cell align=left,align=left,draw=none,fill=none,font=\footnotesize},
			xlabel={computational resources},
			ylabel={$e_{\alpha,\text{rel}}$},
			y tick label style={/pgf/number format/fixed},
			outer sep=0pt,
			font=\footnotesize,
			]
			
			\addplot [color=Set2-A, line width=1pt] table[x=cputime,y=alphaDiff]{data/deblurring/bilevelDeblurFIFB128.txt};
			\addlegendentry{FIFB}
			
			\addplot [color=Set2-B, line width=1pt] table[x=cputime,y=alphaDiff]{data/deblurring/bilevelDeblurFEFB128.txt};
			\addlegendentry{FEFB}
			
			\addplot [color=Set2-D, line width=1pt] table[x=cputime,y=alphaDiff]{data/deblurring/bilevelDeblurAID128.txt};
			\addlegendentry{AID}
			
			\addplot [color=Set2-C, line width=1pt] table[x=cputime,y=alphaDiff]{data/deblurring/bilevelDeblurImplicit128.txt};
			\addlegendentry{implicit}
		\end{axis}
	\end{tikzpicture}
	\caption{Outer problem, $N=128$}
	\label{fig:numerical:outer-deblur128}
\end{subfigure}%
\hfill%
\begin{subfigure}[b]{0.45\textwidth}
	\begin{tikzpicture}
		\begin{axis}[%
			width=\linewidth,
			height=0.75\linewidth,
			xlabel near ticks,
			ylabel near ticks,
			scaled y ticks=false,
			yminorticks=true,
			minor y tick num=1,
			xminorticks=true,
			minor x tick num=3,
			axis x line*=bottom,
			axis y line*=left,
			legend columns=4,
			legend style={at={(0.5, 1.0)}, anchor=south,inner sep=0pt,outer sep=0pt,legend cell align=left,align=left,draw=none,fill=none,font=\footnotesize},
			xlabel={computational resources},
			ylabel={$e_{u,\text{rel}}$},
			y tick label style={/pgf/number format/fixed},
			outer sep=0pt,
			font=\footnotesize,
			]
			
			\addplot [color=Set2-A, line width=1pt] table[x=cputime,y=uDiff]{data/deblurring/bilevelDeblurFIFB128.txt};
			\addlegendentry{FIFB}
			
			\addplot [color=Set2-A, dashed, line width=0.7pt] table[x=cputime,y=SuDiff]{data/deblurring/bilevelDeblurFIFB128Su.txt};
			
			\addplot [color=Set2-B, line width=1pt] table[x=cputime,y=uDiff]{data/deblurring/bilevelDeblurFEFB128.txt};
			\addlegendentry{FEFB}
			
			\addplot [color=Set2-D, line width=1pt] table[x=cputime,y=uDiff]{data/deblurring/bilevelDeblurAID128.txt};
			\addlegendentry{AID}
			
			\addplot [color=Set2-C, line width=1pt] table[x=cputime,y=uDiff]{data/deblurring/bilevelDeblurImplicit128.txt};
			\addlegendentry{implicit}
		\end{axis}
	\end{tikzpicture}
	\caption{Inner problem, $N=128$}
	\label{fig:numerical:inner-deblur128}
\end{subfigure}%
	\caption{Deconvolution performance.
    The graphs correspond to the
    FIFB (\ref*{plot:dc:FIFB}),
    FEFB (\ref*{plot:dc:FEFB}),
    AID (\ref*{plot:dc:AID}), and
    implicit method (\ref*{plot:dc:implicit}). For the inner problem, we additionally plot the relative error of $S_u(\alpha^k)$ for the FIFB (\ref*{plot:dc:FIFB-S_u}).
	} 
	\label{fig:numerical:deblurring}
\end{figure}
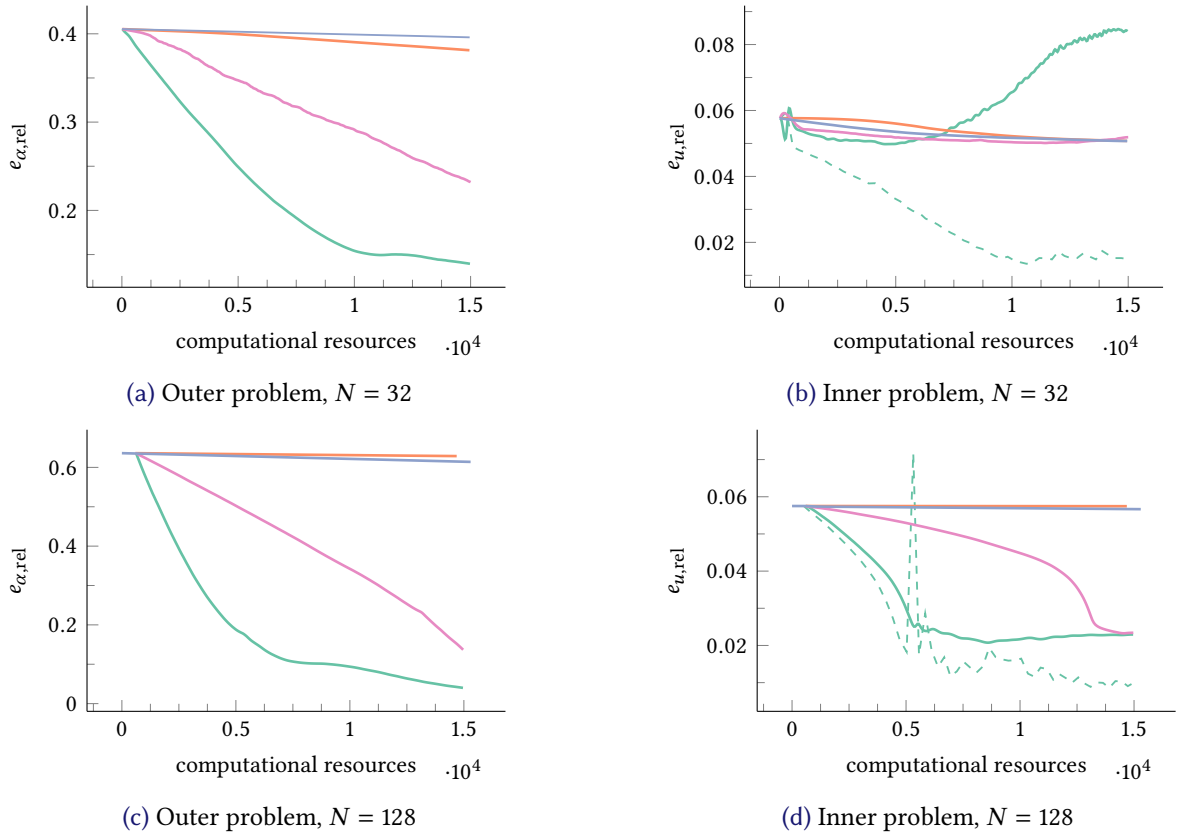

We report performance in \cref{fig:numerical:denoising,fig:numerical:deblurring} and the image data and reconstructions in \cref{fig:denoising,fig:deblurring}.
\Cref{fig:numerical:deblurring} indicates that for deconvolution the FIFB significantly outperforms the other methods.
The outer variable converges much faster than for other evaluated methods despite the fact that the inner variable especially with $N=32$ stays some distance away from $\tilde u$. However, as the  dashed line indicates, the exact solution $S_u(\alpha^k)$ of the inner problem for the corresponding outer iterate, shows clear signs of convergence. (The few “spikes” in the graph for $N=128$ temporarily have the regularisation parameter $\alpha^k_0$ much closer to zero than $\tilde\alpha_0$.) This observation justifies the intuition that the inner problem does not need to be solved to high accuracy to obtain convergence for the outer problem; that such accurate solutions can even be detrimental to convergence.
The exact solution of the adjoint equation in both the implicit method and the FEFB causes them to be too slow to make any meaningful progress.
The denoising experiments of \cref{fig:numerical:denoising} likewise suggest that the FIFB is \emph{initially} the best performing algorithm, although the implicit method and the AID catch up later on the denoising problem.
On the small denoising problem ($N=128$), the implicit method is significantly faster than any oher method.
Overall, and for practical purposes, nevertheless, the FIFB appears to perform the best.

\appendix

\section{Prox-$\sigma$-contractivity}
\label{sec:prox-contractivity}

In the next three theorems we verify prox-$\sigma$-contractivity (\cref{ass:contractivity}) for some common cases.

The next theorem readily extends to $R=\beta\norm{\freevar}_1 + \delta_{[0, \infty)^n}$ on $\R^n$ and products sets $A$ since the proximal mapping of $R$ is independent for each coordinate.

\begin{theorem}[prox-$\sigma$-contractivity of positivity-constrained soft-thresholding]
	\label{thm:contractivity-soft-thresholding}
    Let $\sigma,\beta>0$ and $R=\beta\abs{\freevar}_1 + \delta_{[0, \infty)}$ on $\R$.
    Then $R$ is prox-$\sigma$-contractive at any $\opt\alpha \ge \max\{0, \sigma(q+\beta)\}$ for any $q \in \R^n$ with any factor $0 < C_R < \inv\sigma$ within
    \[
        A
        =
        \left[
            \max\left\{
                0,
                \opt\alpha - \frac{\opt\alpha - \sigma(q + \beta)}{1-\sigma C_R}
            \right\},
            \infty
        \right).
    \]
    In particular, if $\opt\alpha \in \Dom R=[0, \infty)$ and $-q = \beta \in \subdiff R(\opt\alpha)$, then $R$ is locally prox-$\sigma$-contractive at $\opt\alpha$ with any factor $0 < C_R < \inv\sigma$ within
    \[
        A
        =
        \left[
            0,
            \infty
        \right).
    \]
\end{theorem}

\begin{proof}
    We have (see, e.g., \cite{clason2020introduction})
    \[
        D_{\sigma R}(\alpha)
        =
        \prox_{\sigma R}(\alpha - \sigma q) - \alpha =
        \begin{cases}
            -\sigma(q + \beta), & \alpha \ge \sigma(q + \beta), \\
            -\alpha, & \alpha < \sigma(q + \beta).
        \end{cases}
    \]

    We have by assumption $\opt\alpha \ge \sigma(q+\beta)$.
    If also $\alpha \ge \sigma(q + \beta)$, we have $D_{\sigma R}(\alpha) - D_{\sigma R}(\opt\alpha)=0$, which satisfies the required inequality.

    Suppose then that $\alpha < \sigma(q + \beta)$.
    Since $D_{\sigma R}(\opt\alpha)=-\sigma(q+\beta)$, we need to show that
    \begin{equation}
        \label{eq:contractivity-soft-thresholding-expansion1}
        \abs{D_{\sigma R}(\alpha)+\sigma(q+\beta)} \le \sigma C_R (\opt\alpha-\alpha).
    \end{equation}
    We have $D_{\sigma R}(\alpha)=-\alpha$,
    so \eqref{eq:contractivity-soft-thresholding-expansion1} rearranges as
    \[
        -\alpha+\sigma(q+\beta) \le \sigma C_R (\opt\alpha-\alpha).
    \]
    Since $1 > \sigma C_R$, this inequality can be rearranged as the condition $\alpha \ge \opt\alpha - \frac{\opt\alpha - \sigma(q + b)}{1-\sigma C_R}$.
    Any $\alpha \in A$ satisfies this bound.

    Let then $-q = \beta \in \subdiff R(\opt\alpha)$.
    Since $\opt\alpha \ge 0$, we have $\opt\alpha \ge \sigma(q+\beta) = 0$.
    Since $\frac{\opt\alpha - \sigma(q + \beta)}{1-\sigma C_R} = \frac{\opt\alpha}{1-\sigma C_R} \ge \opt\alpha$, the claimed simpler expression for $A$ follows from the general.
\end{proof}

Similarly to the previous result, the restriction $q=0$ in the next theorem on projections to a convex set $C$ forbids stricly complementary cases of $\opt\alpha \in \bd C$, i.e., we cannot have $0 \ne -q \in N_C(\opt\alpha) \defeq \subdiff\delta_C(\opt\alpha)$.

\begin{theorem}[prox-$\sigma$-contractivity of projections]
	\label{thm:contractivity-projection}
    Let $\sigma>0$ and $R=\delta_C$ for a convex and closed $C \subset \R^n$.
    Then $R$ is prox-$\sigma$-contractive at any $\opt\alpha \in C$ for $q=0$ within $A=C$ with any factor $C_R > 0$.
\end{theorem}

\begin{proof}
    We have $\prox_{\sigma R}=\proj_C$ for the Euclidean projection onto $C$.
    Since $\alpha, \opt\alpha \in C = \Dom R$ and $q=0$, we have
    $\alpha = \proj_C(\alpha) = \proj_C(\alpha - \sigma q)$,
    and likewise for $\opt\alpha$.
    The claim is now immediate.
\end{proof}

\begin{example}[ReLu]
    The proximal mapping of $\delta_{[0, \infty)}$ is known as the rectifier linear unit activation function (ReLu). By the above theorem, it is prox-$\sigma$-contractive at any $\opt\alpha \ge 0$ for $q=0$.
\end{example}

\begin{theorem}[prox-$\sigma$-contractivity of smooth functions]
	\label{thm:contractivity-smooth}
    Let $\sigma,\beta>0$ and $R: \R^n \to \R$ be convex with Lipschitz gradient.
    Then $R$ is prox-$\sigma$-contractive at any $\opt\alpha \in \R^n$ for any $q \in \R^n$ within $A=\R^n$ with the factor $C_R=L_{\grad R}$ the Lipschitz factor of $\grad R$.
\end{theorem}

\begin{proof}
    Write $p(\alpha) \defeq \prox_{\sigma R}(t(\alpha))$ and $t(\alpha) \defeq \alpha - \sigma q$.
    According to the definition of the proximal operator,
    $
        0 = \sigma \grad R(p(\alpha)) + p(\alpha) - t(\alpha).
    $
    Hence
    $
        p(\alpha)-\alpha = -\sigma[q + \grad R(p(\alpha))]
    $
    which yields for any $\alpha \in \R^n$, as required,
    \begin{align*}
        \norm{D_{\sigma R}(\alpha)-D_{\sigma R}(\opt\alpha)}
        &
        =
        \norm{[p(\alpha)-\alpha]-[p(\opt\alpha)-\opt\alpha]}
        \\
        &
        =\sigma\norm{\grad R(p(\alpha))-\grad R(p(\opt \alpha))}
        \\
        &
        \le\sigma L_{\grad R}\onorm{p(\alpha)-p(\opt \alpha)}
        \\
        &
        \le\sigma L_{\grad R}\norm{t(\alpha)-t(\opt \alpha)}
        \\
        &
        =
        \sigma L_{\grad R}\norm{\alpha-\opt \alpha}.
        \qedhere
    \end{align*}
\end{proof}

\section{A norm on operators on separable Hilbert spaces}
\label{sec:separable}

We show basic properties of the inner product defined and norm induced by \eqref{def:p-inner_prod}.

\begin{theorem}
    \label{thm:separable:properties}
    Let $U$ be a Hilbert space and $\AlphaSpace$ a separable Hilbert space.
    On $\linear(U; \AlphaSpace)$ define $\oiprod{\freevar}{\freevar}$ and $\onorm{\freevar}$ according to \eqref{def:p-inner_prod}. Then
    \begin{enumerate}[label=(\roman*),nosep]
        \item\label{item:separable:norm-iprod}
        $\oiprod{\freevar}{\freevar}$ is an inner product and $\onorm{\freevar}$ a norm on $\linear(U; \AlphaSpace)$.
        \item\label{item:separable:operator-norm-combo}
        For $M \in \linear(U; U)$ and $p \in \linear(U; \AlphaSpace)$, we have $\onorm{pM} \le \onorm{p}\norm{M}_{\linear(U; U)}$.
        \item\label{item:separable:operator-norm} The operator norm on $\linear(U; \AlphaSpace)$ satisfies $\norm{\freevar}_{\linear(U; \AlphaSpace)} \le \onorm{\freevar}$.
    \end{enumerate}
\end{theorem}
\begin{proof}
    \cref{item:separable:norm-iprod}
    Clearly $\oiprod{\freevar}{\freevar}$ is bilinear and symmetric.
    Also $\onorm{p} = \oiprod{p}{p} \ge 0$ for all $p \in P$.
    To prove that $\onorm{p} > 0$ for $p \ne 0$, we observe that the contrary implies $\norm{p^*\phi_i}=0$ for all $i \in I$. Since $\{\phi_i\}_{i \in I}$ is a basis for $\AlphaSpace$, this implies $p^*=0$, hence $p=0$.

    \ref{item:separable:operator-norm-combo}
    We have
    $
        \onorm{pM}^2
        =\sum_{i \in I} \norm{M^*p^*\phi_i}^2
        \le \sum_{i \in I} \norm{M^*}_{\linear(U; U)}^2 \norm{p^*\phi_i}^2
        = \norm{M}_{\linear(U; U)}^2 \onorm{p}^2.
    $

    \ref{item:separable:operator-norm}
    Let $p \in \linear(U; \AlphaSpace)$.
    Then
    $\norm{p}_{\linear(U; \AlphaSpace)}
    =\norm{p^*}_{\linear(\AlphaSpace; U)}
    =\sup_{\alpha \in \AlphaSpace, \norm{\alpha}=1} \norm{p^*\alpha}_U$.
    Since $\{\phi_i\}_{i \in I}$ is an orthonormal basis for $\AlphaSpace$, we can write $\alpha=\sum_{i \in I} a_i \phi_i$ for some $a_i \in \R$ with $\sum_{i \in I} a_i^2=1$.
    Thus
    $
        \norm{p^*\alpha}_U^2
        =\sum_{i \in I}\sum_{\alpha \in I} a_i a_j \iprod{p^*\phi_i}{p^*\phi_j}_U
        \le
        \sum_{i \in I} \left(\sum_{\alpha \in I} a_j^2\right) \norm{p^*\phi_i}_U^2
        = \onorm{p}^2,
    $
    where the last inequality uses Young's inequality.
    The claim follows.
\end{proof}

\bibliographystyle{jnsao}

\end{document}